\documentclass[reqno]{amsart}
\usepackage{url}

\usepackage[a4paper, left=3cm, textwidth=15cm]{geometry}
\usepackage{amsmath,amsfonts,amsthm,amssymb}
\usepackage[usenames]{color}
\definecolor{citegreen}{rgb}{0,0.6,0}
\definecolor{refred}{rgb}{0.8,0,0}
\usepackage[colorlinks, citecolor=citegreen, linkcolor=refred]{hyperref}

\usepackage{mdframed}
\usepackage[OT2,OT1]{fontenc}

\newcommand{\del}{\partial}

\newtheorem{thm}{Theorem}
\newtheorem{prop}[thm]{Proposition}

\newtheorem{cor}[thm]{Corollary}
\newtheorem{lemma}[thm]{Lemma}

\newtheorem{rmk}[thm]{Remark}

\numberwithin{equation}{section} \numberwithin{thm}{section}
\newcommand{\R}{\mathbb{R}}
\newcommand{\N}{\mathbb{N}}

\newcommand{\C}{\mathbb{C}}

\newcommand{\p}{\partial}

\newcommand{\mint}{\mathop{\int\hspace{-1.05em}{\--}}\nolimits}

\newcommand{\ep}{\varepsilon} 
\newcommand{\ph}{\varphi} 
\newcommand{\al}{\alpha}

\newcommand{\gw}{\omega}

\newcommand{\lap}{\Delta}

\DeclareMathOperator{\Id}{Id}

\DeclareMathOperator{\Div}{div}
\newcommand{\x}{\textbf{x}}
\usepackage{mathrsfs}

\allowdisplaybreaks

\begin{document}
\title[Rigidity of $\varepsilon$-harmonic maps]{Rigidity of $\varepsilon$-harmonic maps of low degree}

\author[J.~H\"orter]{Jasmin H\"orter}
\address[J.~H\"orter]{Department of Mathematics\\ 
Karlsruhe Institute of Technology \\ 
76128 Karlsruhe\\ Germany}
\email{jasmin.hoerter@kit.edu}

\author[T.~Lamm]{Tobias Lamm}
\address[T.~Lamm]{Department of Mathematics\\ 
Karlsruhe Institute of Technology \\ 
76128 Karlsruhe\\ Germany}
\email{tobias.lamm@kit.edu}

\author[M.~Micallef]{Mario Micallef}
\address[M.~Micallef]{Mathematics Institute\\ 
	University of Warwick\\ 
	Coventry CV4 7AL\\ UK}
\email{M.J.Micallef@warwick.ac.uk}

\thanks{The first and second author gratefully acknowledge funding by the Deutsche Forschungsgemeinschaft (DFG, German Research Foundation) – 281869850 (RTG 2229)}
\date{\today}

\begin{abstract} 
In 1981, Sacks and Uhlenbeck introduced their famous $\alpha$-energy as a way to approximate the Dirichlet energy and produce harmonic maps from surfaces into Riemannian manifolds. 
However, the second and third authors together with Malchiodi (\cite{lammmalmic}, \cite{lammmalmic19}) showed that 
for maps between two-spheres this method does not capture every harmonic map. 
They established a gap theorem for $\alpha$-harmonic maps of degree zero and also showed that 
below a certain energy bound $\alpha$-harmonic maps of degree one are rotations. 
We establish similar results for $\varepsilon$-harmonic maps $u_\varepsilon \colon S^2\rightarrow S^2$, 
which are critical points of the $\varepsilon$-energy introduced by the second author in \cite{lamm06}. 
In particular, we similarly show that $\varepsilon$-harmonic maps of degree zero with energy below $8\pi$ are constant 
and that maps of degree $\pm 1$ with energy below $12\pi$ are of the form $Rx$ with $R\in O(3)$. 
Moreover, we construct non-trivial $\varepsilon$-harmonic maps of degree zero with energy $> 8\pi$. 
\end{abstract}

\maketitle
\section{Introduction}
The Dirichlet energy $E(u)$ of a map $u \in W^{1,2}(M,N)$ from 
a smooth, closed two-dimensional Riemannian manifold $(M^2,g)$ to 
a smooth, closed Riemannian manifold $(N^n, h)$ which is isometrically embedded into some $\R^k$ 
is defined as  

\begin{equation}
  E(u) := \frac{1}{2}\int_M |\nabla u|^2 \, dA_M \label{dirichlet}.
\end{equation} 
Critical points of this functional are called harmonic maps and satisfy
\begin{align}\label{cond harm map}
  \Delta u\perp T_uN.
\end{align}
In 1981 Sacks and Uhlenbeck \cite{SacksUhlenbeck} introduced their famous $\alpha$-energy approximation
\begin{align*}
  E_\alpha(u)= \frac{1}{2}\int_M (1+|\nabla u|^2)^\alpha dA_{M}, \qquad \alpha>1.
\end{align*}
Since $\alpha>1$, this functional has better compactness properties than the Dirichlet energy, 
which allowed  Sacks and Uhlenbeck  to show that, as $\alpha\rightarrow 1$, 
a sequence of critical points $(u_\alpha)$ converges to a harmonic map and finitely many bubbles, i.e. non-trivial two-spheres. 
Now one can ask whether every harmonic map can be captured by this procedure and 
the answer is no under suitable energy assumptions. Together with Malchiodi, 
the second and third authors showed (\cite{lammmalmic}, \cite{lammmalmic19}) that 
the only $\alpha$-harmonic maps  $u \colon S^2\rightarrow S^2$ of degree $\pm 1$ 
with energy $E_\alpha$ below $8^\alpha 2\pi$ are of the form $Rx, ~R\in O(3)$ and 
$\alpha$-harmonic maps of degree zero with energy below $6^\alpha 2\pi$ are constant. 
Moreover, they constructed non-constant $\alpha$-harmonic maps of degree zero 
with energy slightly above $6^\alpha 2\pi$, establishing a gap theorem. 

In the following we pose the same question for the fourth order approximation of the Dirichlet energy
\[
  E_\ep (u) = \frac12 \int_{M} (|\nabla u|^2+\ep |\Delta u|^2) \, dA_{M},\qquad \varepsilon > 0.
\]
This approximation was first studied by the second author in \cite{lamm06}, 
where he showed that critical points $u_\varepsilon\in W^{2,2}(M,N^n)$ exist for every $\ep > 0$ and that 
sequences of critical points satisfy the same bubbling picture as the $\al$-harmonic maps studied earlier 
(see Theorem 1.1 in \cite{lamm06}). Further, critical points are smooth and satisfy 
\begin{align*}
  \Delta u - \varepsilon\Delta^2 u 
&= \varepsilon\sum_{i=n+1}^{k} \bigg(\Delta (\langle \nabla u,(d\nu_i \circ u)\nabla u\rangle )  
  + \Div (\langle \Delta u, (d\nu_i\circ u) \nabla u\rangle ) \\ 
&\qquad+ \langle \nabla\Delta u, (d\nu_i\circ u) \nabla u\rangle \bigg) \nu_i\circ u 
  - A(u)(\nabla u, \nabla u),
\end{align*} 
where $\{\nu_i \}_{i=n+1}^k$ is a smooth orthonormal frame of the normal space of $N$ and 
$A$ is the second fundamental form of the embedding $N\hookrightarrow \R^k$. 
In the following we are only interested in maps $u \colon S^2\rightarrow S^2$. Thus the equation simplifies to 
\begin{align*}
  \Delta u - \varepsilon\Delta^2 u 
&= - u|\nabla u|^2 + \varepsilon u\bigg( \Delta|\nabla u|^2 + \Div \langle \Delta u, \nabla u\rangle 
  + \langle \nabla \Delta u, \nabla u\rangle \bigg).
\end{align*}
The degree of every map $u \colon S^2\rightarrow S^2$ is defined by 
\begin{align*}
  \operatorname{deg}(u) := \frac{1}{4\pi} \int_{S^2} J(u)dA_{S^2} 
  \qquad\text{with}\qquad J(u) = u \cdot e_1(u)\wedge e_2(u),
\end{align*}
where $(e_1,e_2)$ is a local oriented orthonormal frame of $TS^2.$ 
For every $u \in W^{2,2}(S^2,S^2)$ with $\deg(u)=1$ we have
\begin{align}\label{E eps abs}\nonumber
  4\pi(1+2\ep)&= \int_{S^2} J(u) \, dA_{S^2} +\frac{\ep}{2\pi} \left(\int_{S^2} J(u) \, dA_{S^2} \right)^2
	\\\nonumber
&\leqslant E(u) + \frac{\ep}{8\pi}\left(\int_{S^2} |\nabla u|^2\, dA_{S^2}\right)^2
	\\\nonumber
&\leqslant E(u) + \frac{\ep}{2} \int_{S^2}|\nabla u|^4\, dA_{S^2} 
	\\
&\leqslant E_\ep (u),
\end{align}
where we used that $\Delta u = (\Delta u)^T - u |\nabla u|^2$ and therefore
\[
  \int_{S^2} |(\Delta u)^T|^2 \, dA_{S^2}+\int_{S^2} |\nabla u|^4 \, dA_{S^2} =\int_{S^2} |\Delta u|^2 \, dA_{S^2}
\]
in the last step. Thus equality holds in (\ref{E eps abs}) if and only if 
$u$ is a harmonic map (third inequality and see (\ref{cond harm map})), 
which is conformal (first inequality) and with constant energy density (second inequality)	
\begin{align*}
  e(u):= \tfrac12 |\nabla u|^2 \equiv 1.
\end{align*}
For every $R\in SO(3)$ and map $u^R(x)= Rx$ we have 
\begin{align}\label{E var rot}
  E_\varepsilon(u^R) = 4\pi + 8\pi\varepsilon.
\end{align}
Hence the rotations are the only minimizers of $E_\varepsilon$ among all maps of degree 1.
Note that it was shown by Wood and Lemaire ((11.5) in \cite{EellsLemaire}) that all harmonic maps between 2-spheres 
are precisely the rational maps and their complex conjugates (i.e., rational in $z$ or $\bar{z}$). 

A rational map $u$ has Dirichlet energy $E(u)=4\pi |\operatorname{deg }(u)|$, 
which is the least energy that a map of this degree can have. 
However one can verify by direct calculation that dilations, which are rational maps of degree one, 
are not critical points of  $E_\varepsilon$ for $\varepsilon > 0.$ 
This is a very special case of the second of the following two main results in this paper, 
whose proofs occupy the next four sections. 

\begin{thm}\label{thm deg zero constant}
	For any $\delta>0$ there exists $\tilde{\varepsilon}>0$ such that the only critical points $u_\varepsilon$ of $E_\varepsilon$ of degree zero which satisfy $E_\varepsilon(u_\varepsilon)\leqslant 8\pi-\delta$ and $\varepsilon\leqslant \tilde{\varepsilon}$ are the constant maps.
\end{thm}

\begin{thm}\label{thm deg one rot}
  For any $\mu > 0$ there exists $\bar{\varepsilon}>0$ such that 
  the only critical points $u_\varepsilon$ of $E_\varepsilon$ of degree $\pm 1$ which satisfy 
  $E_\varepsilon(u_\varepsilon) \leqslant 12 \pi - \mu$ and $\varepsilon\leqslant \bar{\varepsilon}$ 
  are maps of the form $u^R(x)= Rx$ with $R\in O(3)$.
\end{thm}
Note that we have to include reflections if $\deg u = -1$.
The proof of Theorem \ref{thm deg zero constant} follows analogously to \cite{lammmalmic},\cite{lammmalmic19}. 
We use the energy identity for $\varepsilon$-harmonic maps (see Theorem 1.1 in \cite{lamm06}) and 
a result by Duzaar and Kuwert \cite{duzkuw}, which shows that 
the degree of a sequence $(u_\varepsilon)$ is preserved in the limit. 
The gap theorem for $\varepsilon$-harmonic maps with small energy 
(Lemma \ref{Uhlenbeck gap lemma}) concludes the proof.  

To prove Theorem \ref{thm deg one rot}, we use the group of conformal transformations of the sphere, 
which is called the M\"obius group. In section \ref{sec Moebius} we will see that these transformations 
correspond to $M \in PSL(2,\mathbb{C})$ via stereographic projection to the complex plane.
We follow  Malchiodi's and the second and third authors' idea \cite{lammmalmic} and 
apply a M\"obius transformation $M$ to a critical point $u_\varepsilon$. The goal is to show that 
for every $\varepsilon > 0$ small enough, there exists $M\in PSL(2,\mathbb{C})$ such that 
$(u_\varepsilon)_M := u_\varepsilon \circ M$ is equal to the identity. 
Moreover, we show further that this $M$ defines a rotation on the sphere.

In a first step we investigate how $E_\varepsilon(u_M)$ changes as we vary $M$. 
We will see that the transformation relation depends only on the larger eigenvalue $\lambda$ of $MM^*$ 
and it is therefore enough to demonstrate that $\lambda=1$.
To do this we show that critical points $u_\varepsilon$ are 
close to a  M\"obius transformation in the $\sqrt{\varepsilon}W^{2,2}$-norm and 
simultaneously  establish a bound on $\lambda$ (Proposition \ref{p:clomob3}). 
Thanks to the structure of the $\varepsilon$-approximation, 
the derivative of $E_{\varepsilon,\lambda}$ with respect to $\log \lambda$ is easy to calculate. 
This simplifies the proof of the bound on the eigenvalue $\lambda$ significantly 
compared to the $\alpha$-harmonic case (see Proposition 3.1 in \cite{lammmalmic}).  
However, to improve the bound of $u_{\varepsilon, \lambda}$ in the $\sqrt{\varepsilon}W^{3,2}$-norm  
we have to employ rather technical tools and estimates like Sobolev embeddings and the Poincar\'e inequality. 

In the last step we choose a M\"obius transformation $M^*$ 
(not to be confused with the adjoint $M^*$ in the definition of $\lambda$) 
which is optimal in the sense that $E_\varepsilon(u_M)$ attains a minimum at $M^*$ as $M$ varies over $PSL(2,\C)$. 
This optimal property of $M^*$ is then used to show that the corresponding eigenvalue $\lambda^*$ is equal to one. 
To do this, we consider the tangential component (which we will call $\hat{\psi}$ in the proof) of $(u_{M^*} - \Id)$ at the identity 
and decompose it into the eigenspaces of the tangential Laplacian $(\Delta \cdot )^T$ on vector fields of $S^2$. 
Then we adapt the ideas in \cite{lammmalmic} to complete the proof; the key point is that 
the kernel of the Jacobi operator $J_{\varepsilon}$ is the same as 
the kernel of $(\Delta \cdot )^T + 2$, which is the tangent space of the M\"{o}bius group at the identity 
(and which we will call $Z$ in the proof). The optimality of $M^*$ allows us to show that 
the $Z$-component $\hat{\psi}_0$ of $\hat{\psi}$ is much smaller than $\hat{\psi}$; 
indeed, in suitable norms, $\hat{\psi}_0$ will be shown to be of quadratic order in $\hat{\psi}$. 
This is a key ingredient in the proof that $\lambda^* = 1$. 

In the final section of the paper we prove the following theorem which shows that 
the bound in Theorem \ref{thm deg zero constant} is optimal. 

\begin{thm}\label{thm deg zero}
  For every $\delta > 0$ there exists $\varepsilon_0>0$ depending only on $\delta$ such that, 
  if $0 < \varepsilon < \varepsilon_0$, there exists an $\varepsilon$-harmonic map 
  $u_\varepsilon \colon S^2 \rightarrow S^2$ with $\operatorname{deg}(u_\varepsilon)=0$ and 
\begin{align*}
  8\pi < E_\varepsilon(u_\varepsilon) <  8\pi + \delta. 
  \end{align*}
\end{thm}
The map $u_{\varepsilon}$ in the above theorem is constructed by minimising $E_{\varepsilon}$ 
among a suitable class of rotationally symmetric maps.


\section{The M\"obius group}\label{sec Moebius}
First  we turn our attention to maps of degree one. As mentioned in the introduction, 
all harmonic maps between two-spheres projected to the complex plane are rational 
with Dirichlet energy $E(u) = 4\pi|\deg(u)|$. 
Moreover, all rational maps of degree one form a group under composition. 
This group is called the  M\"obius group and we shall now describe those of its features that are relevant to us. 

Let $\hat{\mathbb{C}}=\mathbb{C}\cup \{ \infty\}$. A holomorphic function 
$m\colon \hat{\mathbb{C}}\rightarrow \hat{\mathbb{C}}$ of the form  
\begin{align*}
  m(\xi) = \frac{a\xi+b}{c\xi+d},\qquad\text{with }\enspace  ad-bc=1, \enspace a,b,c,d\in\mathbb{C}, 
\end{align*}
is called a M\"obius transformation. All rational functions of degree one are of this form. 
We write the coefficients $a,b,c,d$ in matrix form
\begin{align*}
  M&= \begin{pmatrix} a&b\\c&d \end{pmatrix} 
  \enspace \text{with}\enspace \det M=1.
\end{align*} 
The group of such matrices is called the special linear group and it is denoted by $SL(2,\mathbb{C})$. 
Note that $M$ and $N$ in $SL(2, \C)$ represent the same rational map of degree one if, and only if, $M = \pm N$. 
Thus, the M\"obius group can be identified with the projective special linear group $PSL(2,\C)$ defined by 
\[ 
PSL(2,\C) := SL(2,\C)/\{\pm I_2\} 
\] 
where $I_2 = \big(\begin{smallmatrix} 1& 0 \\ 0 & 1 \end{smallmatrix}\big)$. 

It is well known that, if we identify $\hat{ \mathbb{C}}$ with $S^2$ via stereographic projection, 
a M\"{o}bius transformation can be expressed as a rotation followed by a dilation followed by another rotation. 
This can be seen from the singular value decomposition of $M$ 
according to which there exists $U,V \in SU(2)$ such that 
\begin{equation} \label{eq:singvalueM} 
  M= UDV^*
\end{equation}
where $D$ is the diagonal matrix whose entries $D_{11}, D_{22} > 0$ are the square roots of the eigenvalues of $M M^*$, 
where $M^*$ is the adjoint of $M$. Since $\det M=1$ it follows that  $\det D=1$ and thus $D_{11}=\frac{1}{D_{22}}$. 
By relabelling $D_{11}$ and $D_{22}$ if necessary, we see that $D$ can be written in the form 
\begin{align}\label{M diag ev}
  D = \begin{pmatrix} \lambda^{\frac12} &0 \\ 0& \lambda^{-\frac12} \end{pmatrix}, \quad \lambda \geqslant 1. 
\end{align}
If $M=D$, the corresponding M\"obius transformation is a dilation 
\begin{align*}
  m_\lambda(\xi):= \lambda\xi. 
\end{align*} 
We shall now explain the well-known fact that elements of $SU(2)$ in the M\"{o}bius group can be identified with rotations. 
We start by viewing $S^2$ as the complex projective line $\C\mathbb{P}^1$ where, 
if $\sim$ denotes the equivalence relation on $\C^2 \setminus \{(0,0)\}$ defined by 
$(\sigma,\tau) \sim (\rho \sigma ,\rho \tau), \quad \rho \neq 0$, then 
$\C\mathbb{P}^1 := (\C^2 \setminus \{(0,0)\}) / \sim$. 
Points in $\C\mathbb{P}^1$ shall be written as $[\gw], \ \gw = (\sigma,\tau) \in \C^2 \setminus \{(0,0\}$. 
We next identify $[\gw]$ with the orthogonal projection $\Pi_{[\gw]}$ of $\C^2$ onto the line spanned by $\gw$. 
Thus $\Pi_{[\gw]}$ is the $2 \times 2$ hermitian matrix of rank 1 and trace equal to 1 defined by 
\[ 
\Pi_{[\gw]} := \frac{\gw \otimes \gw^*}{|\gw|^2} = \frac{1}{|\sigma|^2 + |\tau|^2}\begin{pmatrix} 
  |\sigma|^2 & \sigma \bar{\tau} \\[\jot] 
  \bar{\sigma} \tau & |\tau|^2 \end{pmatrix}. 
\] 
The action of $PSL(2,\C)$ on $S^2$ can then be seen as the following action on the (nonlinear) space 
$H_1$ of $2 \times 2$ hermitian matrices of rank 1 whose trace is equal to 1: 
\begin{equation} \label{eq:MobH1} 
  \Pi_{[\gw]} \mapsto \mu_M(\Pi_{[\gw]}) := \Pi_{[M\gw]} = 
  \frac{M(\gw \otimes \gw^*)M^*}{|M \gw|^2}, \quad M \in SL(2,\C). 
\end{equation} 
Observe that $\mu_M = \mu_{(-M)}$ and therefore, the action $\mu$ just defined is indeed an action of $PSL(2,\C)$. 

In order to identify a M\"obius transformation as a rotation we need to see $S^2$ as a subset of $\R^3$. 
To this end, we identify $\R^3$ with the space $H_0$ of traceless Hermitian $2 \times 2$ matrices: 
\[ 
H_0 := \left\{\begin{pmatrix} z & x + iy \\ x - iy & -z \end{pmatrix} : x,y,z \in \R\right\}. 
\] 
Note that if $H \in H_0$ then $\det H = -(x^2 + y^2 + z^2)$. 
We may map matrices from $H_1$ to $H_0$ by subtracting $\frac12 I_2$ and then, 
for normalisation purposes, multiplying by 2. Thus, 
\[ 
S^2 := 2\frac{\gw \otimes \gw^*}{|\gw|^2} - I_2 = \frac{1}{|\sigma|^2 + |\tau|^2}\begin{pmatrix} 
  |\sigma|^2 - |\tau|^2 & 2 \sigma \bar{\tau} \\[\jot] 
  2 \bar{\sigma} \tau & |\tau|^2 - |\sigma|^2\end{pmatrix}, \quad [\gw] = [(\sigma,\tau)] \in \C\mathbb{P}^1. 
\] 
For $\tau \neq 0$ we have $[(\sigma,\tau)] = [(\xi,1)]$ where $\xi = \sigma/\tau$ and the corresponding point in $\R^3$ is then 
\[ 
\left( \frac{\xi + \bar{\xi}}{|\xi|^2 + 1}\,,\, \frac{-i(\xi - \bar{\xi})}{|\xi|^2 + 1}\,,\, \frac{|\xi|^2 -1}{|\xi|^2 + 1} \right) 
\] 
which is the image under the inverse of the stereographic projection from the North Pole of the point $\xi \in \C$. 
In order to preserve the traceless condition on matrices in $H_0$, 
the action $\mu$ of $PSL(2,\C)$ on $H_1$ described in \eqref{eq:MobH1} can only be transferred to $H_0$ 
for those matrices $M$ which are unitary and which therefore (since their determinant is 1) must belong to $SU(2)$. 
Now for $M \in SU(2)$ and $H \in H_0$ we have $\det(\mu_M H) = \det(M H M^*) = \det H$, 
which shows that $\mu_M$ is an isometry of $\R^3$ for all $M \in SU(2)$. 
Using the fact that $SU(2)$ is connected and that $\mu_{I_2}$ is the identity on $\R^3$ 
we see that $\mu_M \in SO(3)$ for all $M \in SU(2)$. As promised, we can now interpret \eqref{eq:singvalueM} 
as saying that a M\"{o}bius transformation can be expressed as a rotation followed by a dilation followed by another rotation. 
In particular, $M \in PSL(2,\C)$ represents a rotation if, and only if, the larger eigenvalue $\lambda$ of $M M^*$ is equal to 1.

Let us go back to our original problem and use  the M\"obius transformations in the following way: Let $u_\varepsilon\in W^{1,2}(S^2,S^2)$ be a degree one critical point of $E_\varepsilon$.
The idea is to compose $u_\varepsilon$ with a M\"obius transformation $M$ and show that $(u_\varepsilon)_M$ is close to the identity map $\Id:S^2 \rightarrow S^2$ if $E(u_\varepsilon) < 4\pi + 8\pi\varepsilon + \mu$ and $\mu > 0$ is sufficiently small. If we are able to show that there exists $M\in PSL(2,\mathbb{C})$ such that $(u_\varepsilon)_M$ is actually equal to $\Id$, then $u_\varepsilon$ itself has to be a M\"obius transformation. Moreover,  if $M\in SU(2)$ then $u_\varepsilon$ is a rotation.

In a first step we investigate how $E_\varepsilon$ transforms if we apply $u_M$. To do this we work in stereographic coordinates and consider $u:\hat{\mathbb{C}} \rightarrow S^2$.  The Riemannian metric on $S^2$ in stereographic coordinates is given by $g_{ij}= \frac{4}{(1+|\xi|^2)^2}\delta_{ij}$. For $\xi \in \hat{ \mathbb{C}}$ we have 
\begin{align*}
	|\nabla_{S^2} u|^2(\xi)= \frac{(1+|\xi|^2)^2}{4}|\nabla _\mathbb{C} u|^2(\xi)\qquad\text{and}\qquad |\Delta_{S^2} u|^2(\xi)=\frac{(1+|\xi|^2)^4}{16}|\Delta_{\mathbb{C}} u|^2(\xi),
\end{align*}
where $\nabla_{\mathbb{C}}$ is the gradient on $\mathbb{C}$ and $\Delta_{\mathbb{C}}$ the Laplacian on $\mathbb{C}$ with the flat metric on both the domain and the target. The area element is given by 
\begin{align*}
	dA_{S^2}= \frac{4}{(1+|\xi|^2)^2}dA_{\mathbb{C}},
\end{align*}
with $dA_{\mathbb{C}}= \frac{\sqrt{-1}}{2} d\xi\wedge d\bar{\xi}$ the Euclidean area element on $\mathbb{C}$.
We define $u_M$ by
\begin{align*}
	u_M(\xi)=u(M\xi) = u \left(\frac{a\xi+b}{c\xi+d}\right).
\end{align*}
Using the above and the fact that $M\xi$ is harmonic we have 
\begin{align*}
	|\Delta_{S^2} u_M|^2 (\xi) &= \frac{(1+|\xi|^2)^4}{16}|\Delta_{\mathbb{C}} u_M|^2(\xi)
	\\
	&= \frac{(1+|\xi|^2)^4}{16} \left|\frac{d}{d\xi} \left(\frac{a\xi+ b}{c\xi+d}\right)\right|^4 |\Delta_{\mathbb{C}} u|^2(M\xi)
	\\
	&= \frac{(1+|\xi|^2)^4}{16} \frac{1}{|c\xi +d|^8} |\Delta_{\mathbb{C}} u|^2(M\xi)
	\\
	&= \frac{(1+|\xi|^2)^4}{|c\xi+d|^8 (1+|M\xi|^2)^4}|\Delta_{S^2} u|^2(M\xi).
\end{align*}
With the singular value decomposition in (\ref{M diag ev}) we have 
\begin{align*}
	|c\xi +d|^2 (1+|M\xi|^2) = |a\xi+b|^2 + |c\xi+d|^2= \left|\left(\begin{matrix}
		a&b\\c&d
	\end{matrix}\right) \left( \begin{matrix}
		\xi \\1
	\end{matrix}\right)\right|^2= \left| \left(\begin{matrix}
		\lambda^\frac{1}{2} &0\\0&\lambda^{-\frac{1}{2}}
	\end{matrix}\right) \left(\begin{matrix}
		\xi\\1
	\end{matrix}\right)\right|^2
	= \frac{\lambda^2 |\xi|^2 +1}{\lambda}
\end{align*}
and therefore 
\begin{align}\label{transfo Delta u}
	|\Delta_{S^2} u_M|^2 (\xi) = \frac{\lambda^4(1+|\xi|^2)^4}{(1+\lambda^2 |\xi|^2)^4}|\Delta_{S^2} u|^2(M\xi).
\end{align}
Analogously we get 
\begin{align}\label{transf e(u)}
	|\nabla_{S^2} u_M|^2(\xi)= \frac{\lambda^2 (1+|\xi|^2)^2}{(1+\lambda^2|\xi|^2)^2}|\nabla_{S^2} u|^2(M\xi).
\end{align}
Note that the transformation relation depends only on the eigenvalue $\lambda$.  Hence it is enough to restrict our attention in the following to dilations $m_\lambda$. To show that $M\in SU(2)$ it suffices to show that $\lambda=1$. We set 
\begin{align*}
	u(m_\lambda(\xi))= u(\lambda\xi)=:u_\lambda(\xi)
\end{align*}
and
\begin{align*}
	\chi_\lambda(\xi)=\frac{(1+\lambda^2 |\xi|^2)^2}{\lambda^2(1+|\xi|^2)^2}.
\end{align*}
With (\ref{transfo Delta u}) and (\ref{transf e(u)}) we have 
\begin{align*}
	|\nabla_{S^2} u|^2(\lambda\xi)= \chi_\lambda(\xi) |\nabla_{S^2} u_\lambda |^2(\xi)\qquad\text{and} \qquad |\Delta_{S^2} u|^2(\lambda\xi) = \chi_\lambda^2(\xi) |\Delta_{S^2} u_\lambda|^2(\xi).
\end{align*}
Applying all of this to $E_\varepsilon$ we get 
\begin{align}\label{e var E lambda var}\nonumber
	E_\varepsilon(u)&=\frac{1}{2} \int_{\mathbb{C}} \left(|\nabla_{S^2}u|^2(\xi) + \varepsilon |\Delta_{S^2} u|^2(\xi) \right) \frac{4}{(1+|\xi|^2)^2} dA_{\mathbb{C}}(\xi)
	\\\nonumber
	&= \frac{1}{2}\int_\mathbb{C} \left(|\nabla_{S^2} u|^2(\lambda\xi) +\varepsilon|\Delta_{S^2} u|^2(\lambda\xi)\right) \frac{4\lambda^2}{(1+|\lambda\xi|^2)^2}d A_{\mathbb{C}}(\xi)
	\\\nonumber
	&= \frac{1}{2}\int_{\mathbb{C}}\left(\chi_\lambda(\xi) |\nabla_{S^2} u_\lambda |^2(\xi)+\varepsilon\chi_\lambda^2(\xi) |\Delta_{S^2} u_\lambda|^2(\xi) \right)\frac{4\lambda^2}{(1+|\lambda\xi|^2)^2}d A_{\mathbb{C}}(\xi)
	\\\nonumber
	&= \frac{1}{2}\int_{\mathbb{C}} \left( |\nabla_{S^2} u_\lambda |^2(\xi)+\varepsilon\chi_\lambda(\xi) |\Delta_{S^2} u_\lambda|^2(\xi) \right) \frac{4}{(1+|\xi|^2)^2} dA_{\mathbb{C}}(\xi)
	\\\nonumber
	&=\frac{1}{2}\int_{S^2}\left( |\nabla_{S^2} u_\lambda|^{2}+\varepsilon\chi_\lambda|\Delta_{S^2} u_\lambda|^2\right) dA_{S^2}
	\\
	&=: E_{\varepsilon,\lambda}(u_\lambda).\text{\footnotemark}
\end{align}
\footnotetext{Note that $E_{\varepsilon,\lambda}$ differs from $E_\varepsilon$ only by the factor of $\chi_\lambda$ that multiplies the $|\Delta_{S^2} u_\lambda|^2$ term. This factor $\chi_\lambda$ is important because it measures the lack of conformal invariance of the integral $\int_{S^2} |\Delta_{S^2} u|^2 \, dA_{S^2}$.}
Hence $u$ is a critical point of $E_\varepsilon$ if and only if $u_\lambda$ is a critical point of $E_{\varepsilon,\lambda}$. Since $E_{\varepsilon,\lambda}(u_\lambda)=E_{\varepsilon,\lambda^{-1}}(u_{\lambda^{-1}})$ we will assume from now on that $\lambda\geqslant 1.$
In the following we omit the subscript and write $\nabla=\nabla_{S^2},~ \Delta=\Delta_{S^2}$.
An easy calculation (see \cite{ChangWangYang} Proposition 1.1) shows

\begin{prop}
	Let $\varepsilon>0$. Every critical point $v\in W^{2,2}(S^2,S^2)$ of $E_{\varepsilon,\lambda}$ satisfies the following Euler-Lagrange equation 
	\begin{align}\label{EL eps}
		-\Delta v +\ep \Delta (\chi_\lambda \Delta v) &= v \bigg(|\nabla v|^2-\ep \Delta (\chi_\lambda |\nabla v|^2) - 2\ep \Div \langle \chi_\lambda \Delta v, \nabla v\rangle +\ep \chi_\lambda |\Delta v|^2 \bigg).
	\end{align}
	
\end{prop}


\section{Closeness to the M\"obius group}

In this section we consider critical points of $E_\varepsilon$ of degree 1 whose $\varepsilon$-energy lies below $4\pi(1+2\varepsilon)+\mu$,  $\mu>0$ small. Our aim is to show that these maps are $W^{1,2}$-close to a M\"obius transformation. 
We start with a Lemma collecting various properties of critical points of $E_{\ep,\lambda}$.
\begin{lemma}\label{propenergy}
    Let $v\in W^{2,2}$ be a critical point of $E_{\ep,\lambda}$, then we have
    \begin{align}
      0= 	\frac{d}{d\, \log \lambda} E_{\ep,\lambda}(v) = 
	\ep \int_{S^2} \chi_\lambda z(\lambda \cdot) |\Delta v|^2 \, dA_{S^2},	\label{stat1} \end{align}
	where $z(\xi)=\frac{|\xi|^2-1}{|\xi|^2+1}$. Moreover, we have
	\begin{align}
	 E_{\ep,\lambda}(\Id)   =4\pi\big(1+ \frac{2 \ep}{3}(\lambda^2+1+\lambda^{-2})\big), \label{eq:Eepl} 
	\end{align}
	and there exists a constant $C>0$ so that
	\begin{align}
	  	C \ep (\lambda^2 - 1) &\leqslant \frac{d}{d\, \log \lambda}  E_{\ep,\lambda}(\Id)-\frac{d}{d\, \log \lambda} E_{\ep,\lambda}(v) 
		\notag
		\\ 
		&\leqslant  \sqrt{\ep} \| \sqrt{\chi_{\lambda}}\Delta (v-\Id)\|_{L^2(S^2)} 
		\sqrt{\ep} \big(\| \sqrt{\chi_{\lambda}}\Delta v\|_{L^2(S^2)} + \| \sqrt{\chi_{\lambda}}\Delta \Id\|_{L^2(S^2)}\big).  \label{eq:lsqbd} 
	\end{align}
    \end{lemma}
    \begin{proof}
        	We start by calculating as in section 5 of \cite{lammmalmic}
	\begin{align*}
		\log(\chi_\lambda(\xi))&= 2\log(1+\lambda^2 |\xi|^2)- 2\log \lambda -2\log (1+|\xi|^2),
		\\
		\frac{d}{d\lambda} \log (\chi_\lambda(\xi)) &= \frac{4\lambda|\xi|^2}{1+\lambda^2|\xi|^2}-\frac{2}{\lambda},
		\\
		\frac{d}{d\log \lambda}\log(\chi_\lambda(\xi))&= \frac{2(\lambda^2|\xi|^2-1)}{(\lambda^2|\xi|^2+1)},
	\end{align*}
	to obtain
	\[
	\frac{d}{d\, \log \lambda} E_{\ep,\lambda}(v) = 
	\ep \int_{S^2} \chi_\lambda z(\lambda \cdot) |\Delta v|^2 \, dA_{S^2},	\] 
	where $z(\xi)=\frac{|\xi|^2-1}{|\xi|^2+1}$. Since $v$ is a critical point of $E_{\varepsilon, \lambda}$ we have $E_{\varepsilon,\lambda}'(v)=0$.
	Using $E_{\varepsilon,\tau}(v)=E_{\varepsilon,\lambda}(v_{\lambda\tau^{-1}}) $ we get
	\begin{align*}
		\frac{d}{d\log \tau}E_{\varepsilon, \tau} (v)|_{\tau=\lambda}=\left( \tau \frac{d}{d\tau} E_{\varepsilon,\lambda}(v_{\lambda\tau^{-1}})\right)\bigg|_{\tau=\lambda}= E_{\varepsilon,\lambda}'(v)\left( \tau \frac{d}{d\tau }v_{\lambda\tau^{-1} }\right)\bigg|_{\tau=\lambda}
	\end{align*}
	and thus 
	\begin{align*}
		\frac{d}{d\, \log \lambda} E_{\ep,\lambda}(v) = 0,
	\end{align*}
	which implies \eqref{stat1}.
	Further
	\begin{align*} 
	 E_{\ep,\lambda} (\Id)&= 4\pi+2\ep \int_{S^2} \chi_\lambda \, dA_{S^2} \notag \\
		&= 4\pi+ 2\ep \int_{\mathbb{C}} \frac{(1+\lambda^2|\xi|^2)^2}{\lambda^2(1+|\xi|^2)^2} \frac{4}{(1+|\xi|^2)^2} \, dA_{\mathbb{C}}(\xi) 
		\notag \\
		&= 4\pi+16\pi \ep \int_0^\infty \frac{(1+\lambda^2 r^2)^2}{\lambda^2(1+r^2)^2} \frac{r}{(1+r^2)^2} \, dr 
		\notag \\
		&= 4\pi +8\pi \ep \frac{\lambda}{\lambda^2-1} \int_{\lambda^{-1}}^\lambda w^2 \, dw \notag \\
		&= 4\pi\big(1+ \frac{2 \ep}{3}(\lambda^2+1+\lambda^{-2})\big), 
	\end{align*} 
	where we used the substitution $w=\frac{1+\lambda^2r^2}{\lambda(1+r^2)}$. Differentiating this explicit expression for $E_{\ep,\lambda}(\Id)$ with respect to $\log \lambda$ yields 
	\begin{align} 
		\frac{d}{d \log \lambda} E_{\ep,\lambda}(\Id) &= \frac{16\pi \ep}{3}(\lambda^2-\lambda^{-2}) \notag 
		\\ 
		& = \frac{16\pi \ep}{3}(\lambda^2 - 1) \frac{\lambda^2 + 1}{\lambda^2} \notag
		\\ 
		& \geqslant \frac{16\pi\ep }{3}  (\lambda^2 - 1). \label{eq:lowerlabd} 
	\end{align} 
	Since $\|z\|_{L^\infty(S^2)}\leqslant 1$, we conclude that 
	\begin{align*}
		C \ep (\lambda^2 - 1) &\leqslant \frac{d}{d\, \log \lambda}  E_{\ep,\lambda}(\Id)-\frac{d}{d\, \log \lambda} E_{\ep,\lambda}(v) 
		\notag
		\\ 
		&= \ep \int_{S^2} \chi_\lambda z(\lambda \cdot) (|\Delta \Id|^2 - |\Delta v|^2) \, dA_{S^2} \notag 
		\\ 
		&\leqslant  \sqrt{\ep} \| \sqrt{\chi_{\lambda}}\Delta (v-\Id)\|_{L^2(S^2)} 
		\sqrt{\ep} \big(\| \sqrt{\chi_{\lambda}}\Delta v\|_{L^2(S^2)} + \| \sqrt{\chi_{\lambda}}\Delta \Id\|_{L^2(S^2)}\big). 
	\end{align*} 
    \end{proof}
Inequality \eqref{eq:lsqbd} is our main tool for measuring the deviation of $\lambda$ from 1. We take the first step in this direction in the next proposition. 
\begin{prop}\label{p:clomob3} 
	For any $\delta>0$ there exists $\mu > 0$ such that, 
	if $0 < \ep \leqslant 1$ and if $E_\ep(u) \leqslant 4\pi(1+2\ep) + \mu$, 
	where $u$ is a critical point of $E_\ep$ of degree $1$, 
	then there exists $M \in PSL(2,\mathbb{C})$ such that 
	\begin{align}\label{eq:delta-close2}
		\left\|  \nabla (u_M - \Id) \right\|_{L^2(S^2)} + \sqrt{\ep} \left\| \sqrt{\chi_\lambda} \, \Delta (u_M-\Id) \right\|_{L^2(S^2)} \leqslant \delta. 
	\end{align}
	Furthermore, there exists a fixed constant $C>0$ such that
	if $\lambda \geqslant 1$ is the largest eigenvalue of $MM^*$ (see (\ref{M diag ev})), then 
	\begin{equation} \label{eq:alphalambd2} 
		\ep (\lambda^2-1) \leqslant C \delta. 
	\end{equation} 
\end{prop} 

\begin{proof} 
	We prove \eqref{eq:delta-close2} by contradiction using the energy identity in \cite{lamm06}. If \eqref{eq:delta-close2} were not true, 
	then we could find a sequence $\mu_n \downarrow 0$, a sequence $\ep_n \in (0,1]$, 
	a sequence $u_n \in W^{2,2}(S^2,S^2)$ of critical points of $E_{\ep_n}$ of degree one, 
	with $E_{\ep_n}(u_n) \leqslant 4\pi(1+2\ep_n)+ \mu_n$ 
	and $\delta > 0$ such that 
	\begin{equation} \label{eq:delta}
		\left\| \nabla \big((u_n)_M - \Id \big)  \right\|_{L^2(S^2)} + 
		\sqrt{\ep_n} \left\| \sqrt{\chi_\lambda}  \Delta \big((u_n)_M - \Id \big)  \right\|_{L^2(S^2)} 
		> \delta 
	\end{equation} 
	for all $M \in PSL(2,\mathbb{C})$. Now we have to consider two cases: 
	\\[8pt]
	\underline{$\ep_n \to 0$ :} 
	\\[8pt]
	There exists $n_0\in \N$ large enough such that $\varepsilon_n<\frac{1}{4}$ and $\mu_n<\frac{1}{2}~\forall n\geqslant n_0$. Then $E_{\varepsilon_n}(u_n)$ is uniformly bounded by $6\pi+\frac{1}{2}$ for all $n\geqslant n_0$.
	By Theorem 1.1 in \cite{lamm06} and Theorem 2 in \cite{duzkuw}, $(u_{n})$ converges to a harmonic map $u^*$ and finitely many non-trivial harmonic two-spheres $u^i:S^2\rightarrow S^2$ with $\deg (u^*)+\sum_{i=1}^{k}\deg(u^i)=1$. 
	With the result of Lemaire and Wood mentioned in the introduction, $u^*$, $u^i$ are rational maps with energy $E(u^*)= 4\pi|\deg (u^*)|, ~ E(u^i)= 4\pi |\deg (u^i)|$. Since
	\begin{align*}
		4\pi =\lim_{n\rightarrow \infty} E_{\varepsilon_n}(u_{n})= E(u^*)+\sum_{i=1}^{k}E(u^i)= 4\pi( |\deg(u^*)|+ \sum_{i=1}^{k}|\deg (u^i)|) 
	\end{align*}
	$u^*$ is either a rational map with $\deg(u^*)=1$ and $k=0$, or $u^*$ is a constant map, $k=1$ and $u^1:S^2 \rightarrow S^2$ is a harmonic map of degree one. 
	In the first case, $u^*=m^*$ with some corresponding $M^*\in PSL(2,\mathbb{C})$ which is a contradiction to (\ref{eq:delta}). 
	\\
	\\In the second case, energy concentrates in a small neighborhood of some $x_0\in S^2$ and $u_n$ converges to a constant map away from $x_0$. Without loss of generality let $x_0$ be the south pole $S$. Let $\sigma_n\downarrow 0$ and $D_n$ be a sequence of small disks around $S$ such that the energy on $S^2 \setminus D_n$ is smaller than $\sigma_n$. We project $D_n$ onto the complex plane. Then $\Pi(D_n)=B_{r_n}(0)$, the complex ball with radius $r_n$ and $r_n\rightarrow 0$.  We perform a blow-up as in \cite{lamm06} and define 
	\begin{align*}
		v_n\colon \hat{ \mathbb{C}}\rightarrow S^2,\qquad v_n(\xi)= u_n\circ \Pi^{-1}\left(\frac{\xi}{r_n}\right).
	\end{align*}
	Note that this rescaling corresponds to dilations $m_{\lambda_n}$ on the sphere with $\lambda_n=\frac{1}{r_n}$ and $D_n$ gets mapped to the lower hemisphere. 
	$v_n$ is a critical point of $E_{\tilde{\varepsilon}_n}$ with $\tilde{\varepsilon}_n= \frac{\varepsilon_n}{r_n^2}$. By Lemma 3.1 in \cite{lamm06} we have 
	\begin{align*}
		v_n\rightarrow v^*\qquad \text{in }C^m_{loc}(\mathbb{C}, S^2)\qquad\forall m\in \N,
	\end{align*}
	where $v^*\colon\mathbb{C} \rightarrow S^2$ is a non-trivial harmonic map. 
	With the point removability result of Sacks and Uhlenbeck \cite{SacksUhlenbeck} we can lift $v^*$ to a harmonic map from $S^2$ to $S^2$ with corresponding  $M^*\in PSL(2,\mathbb{C})$ such that
	\begin{align*}
		0\leftarrow&||\nabla (v_n- v^*)||_{L^2(S^2)}+ \sqrt{\varepsilon_n} \lambda_n||\Delta (v_n-v^*)||_{L^2(S^2)} 
		\\
		&\geqslant ||\nabla (v_n- v^*)||_{L^2(S^2)}+ \sqrt{\varepsilon_n} || \sqrt{\chi_{\lambda_{n}}}\Delta (v_n-v^*)||_{L^2(S^2)} 
		\\
		&= ||\nabla (u_n- (v^*)_{M^{-1}_{\lambda_n}})||_{L^2(S^2)}+ \sqrt{\varepsilon_n} ||\Delta (u_n-(v^*)_{M^{-1}_{\lambda_n}})||_{L^2(S^2)} 
		\\
		&=  ||\nabla ((u_n)_{M_{\lambda_n}(M^*)^{-1}}  - \Id)||_{L^2(S^2)}+ \sqrt{\varepsilon_n} ||\sqrt{\chi_{M_{\lambda_n}(M^*)^{-1}}}\Delta ((u_n)_{M_{\lambda_n}(M^*)^{-1}}-\Id)||_{L^2(S^2)}.
	\end{align*}
	\underline{$\ep_n \to \ep_\infty \in (0,1]$ :}
	\\[8pt]
	Here we have, at least for $n$ large enough, 
	a uniform $W^{2,2}$-bound for the sequence $u_n$. 
	With the regularity results for $\varepsilon$-harmonic maps (see  \cite{Jasi}), \cite{lamm06}, \cite{Lamm10}) we conclude that $u_n$ converges strongly in $W^{2,2}$ 
	to a limiting map $u_\infty$ which is a critical point of $E_{\ep_\infty}$ 
	and which satisfies
	\[
	E_{\ep_\infty}(u_\infty)=4\pi(1+2\ep_\infty).
	\] 
	By (\ref{E var rot}) this implies that $u_\infty$ is a rotation, contradicting \eqref{eq:delta}.
	\\\\
	To establish \eqref{eq:alphalambd2}
	we set $v := u_M$ and use Lemma \ref{propenergy} to get
	\begin{align*}
		C \ep (\lambda^2 - 1) 
		&\leqslant  \sqrt{\ep} \| \sqrt{\chi_{\lambda}}\Delta (v-\Id)\|_{L^2(S^2)} 
		\sqrt{\ep} \big(\| \sqrt{\chi_{\lambda}}\Delta v\|_{L^2(S^2)} + \| \sqrt{\chi_{\lambda}}\Delta \Id\|_{L^2(S^2)}\big). 
	\end{align*} 
	By assumption, 
	\begin{align*} 
		4 \pi (1+2\ep) + \mu \geqslant E_\ep(u) = E_{\ep,\lambda}(u_M) &= E_{\ep,\lambda}(v)  \geqslant 4 \pi + \frac{\ep}{2} \int_{S^2}\chi_\lambda |\Delta v|^2 \, dA_{S^2}, 
	\end{align*} 
	where we used 
	\[ 
	\frac12 \int_{S^2} |\nabla v|^2 \, dA_{S^2} \geqslant 4 \pi 
	\] 
	in the second inequality, which holds because $\deg(v)=1$. Thus
	\begin{equation} \label{eq:Deltav_ubd} 
		\ep \| \sqrt{\chi_{\lambda}}\Delta v\|_{L^2(S^2)}^2 \leqslant 16 \pi \ep + 2 \mu. 
	\end{equation} 
	By the triangle inequality
	\begin{equation} \label{eq:DeltaIdubd} 
		\sqrt{\ep} \| \sqrt{\chi_{\lambda}}\Delta \Id\|_{L^2(S^2)} \leqslant	\sqrt{\ep} \big( \| \sqrt{\chi_{\lambda}}\Delta(\Id - v)\|_{L^2(S^2)} 
		+ \| \sqrt{\chi_{\lambda}}\Delta v\|_{L^2(S^2)} \big). 
	\end{equation} 
	Using \eqref{eq:delta-close2}, \eqref{eq:Deltav_ubd} and \eqref{eq:DeltaIdubd} in 
	\eqref{eq:lsqbd}, we get 
	\[ 
	\ep (\lambda^2 - 1) \leqslant C \delta. 
	\] 
\end{proof}


\section{Improved bounds on $\lambda$}\label{sec eps W32 closeness}
Next we want to improve the $W^{1,2}$-closeness result from the previous section and get a better  bound on the eigenvalue $\lambda$. 

\begin{prop}
	Suppose $v$ is a critical point of $E_{\varepsilon, \lambda}.$ Setting $\psi:= v-\Id$ we have 
	\begin{align}\label{EL psi}
		\Delta \psi +\psi|\nabla \psi|^2 +2\psi\langle \nabla \psi , \nabla \Id\rangle + 2\psi + \Id |\nabla \psi|^2 + 2 \Id \langle \nabla \psi , \nabla \Id \rangle = \varepsilon\sum_{j=1}^{3}\Psi_j(\psi,\Id)
	\end{align}
	with 
	\begin{align*}
		\Psi_1(\psi,\Id)&= \chi_\lambda\bigg[ \Delta^2 \psi -4\psi +(\psi+\Id)\bigg(4 \langle \nabla \Delta \psi, \nabla \psi \rangle  +4\langle \nabla \Delta \psi, \nabla \Id \rangle +|\Delta \psi|^2 + 2 |\nabla^2 \psi|^2
		\\
		&\qquad+ 4 \langle\nabla^2 \psi, \nabla^2 \Id \rangle-4\langle \Delta \psi, \Id \rangle - 8 \langle \nabla \psi, \nabla \Id \rangle\bigg)\bigg],
		\\
		\Psi_2(\psi,\Id)&= \nabla_i\chi_\lambda\bigg[ 2\nabla_i\Delta \psi -4\nabla_i\Id + (\psi+\Id) \bigg( 4\langle \nabla_i\nabla\psi, \nabla\psi\rangle + 2 \langle \Delta \psi, \nabla_i\psi\rangle+ 4 \langle \nabla_i\nabla \psi, \nabla \Id \rangle 
		\\
		&\qquad  + 2 \langle \Delta \psi, \nabla_i\Id \rangle + 4 \langle \nabla \psi, \nabla_i\nabla\Id \rangle -4 \langle \nabla_i\psi, \Id \rangle \bigg)\bigg],
		\\
		\Psi_3(\psi,\Id)&=\Delta \chi_\lambda\bigg[ \Delta \psi + 2 \psi +(\psi+\Id )\bigg( |\nabla \psi|^2 + 2 \langle \nabla \psi, \nabla \Id \rangle \bigg)\bigg].
	\end{align*}
\end{prop}
\begin{proof}
We use (\ref{EL eps}) and replace $v$ with $\psi+\Id$. Note that $\Delta \Id = -2\Id$, $|\Id|^2=1$ and $|\nabla \Id|^2=2$. Then we have 
	\begin{align}\label{EL psi proof}\nonumber
		-\Delta \psi-\Delta \Id -(\psi +\Id) |\nabla \psi +\nabla \Id|^2  &= -\varepsilon \Delta \left(\chi_\lambda (\Delta \psi + \Delta \Id)\right)
		\\\nonumber
		&\quad  + \varepsilon (\psi+\Id)\chi_\lambda |\Delta \psi + \Delta \Id|^2\\\nonumber
		&\quad -\varepsilon(\psi +\Id)  \Delta\left( \chi_\lambda |\nabla \psi +\nabla \Id|^2\right)
		\\\nonumber
		&\quad - 2 \varepsilon(\psi+\Id) \Div \langle \chi_\lambda(\Delta \psi+\Delta \Id), \nabla \psi +\nabla \Id\rangle
		\\\nonumber
		\Leftrightarrow
		\Delta \psi  +\psi |\nabla \psi|^2 +2\psi +2\psi\langle\nabla \psi, \nabla \Id\rangle&+\Id |\nabla \psi|^2+2\Id \langle\nabla \psi, \nabla \Id\rangle
		\\\nonumber
		& = \varepsilon \Delta \left(\chi_\lambda (\Delta \psi + \Delta \Id)\right)
		\\\nonumber
		&\quad  - \varepsilon (\psi+\Id)\chi_\lambda |\Delta \psi + \Delta \Id|^2\\\nonumber
		&\quad +\varepsilon(\psi +\Id)  \Delta \left(\chi_\lambda |\nabla \psi +\nabla \Id|^2\right)
		\\
		&\quad + 2 \varepsilon(\psi+\Id) \Div \langle \chi_\lambda(\Delta \psi+\Delta \Id), \nabla \psi +\nabla \Id\rangle
	\end{align}
The claim now follows from a direct computation (see \cite{Jasi}, Proposition $4.3.1$ for details).
\end{proof}

Before we get to the next lemma note that we can estimate $\chi_\lambda$ and its derivatives in terms of $\lambda$. Note that we assumed $\lambda\geqslant 1$. It is easy to see that $|\chi_\lambda|\leqslant \lambda^2$. Additionally
\begin{align*}
	\p_i\chi_\lambda(\xi)&= \frac{4\xi_i(1+\lambda^2 |\xi|^2)(\lambda^2-1)}{\lambda^2 (1+|\xi|^2)^3}
\end{align*}
and  with $g_{ij}= \frac{4}{(1+|\xi|^2)^2}\delta_{ij}$
\begin{align}\label{Dchi abs}
	|\nabla_{S^2}\chi_\lambda|= \sqrt{g^{ij} \p_i\chi_\lambda\p_j \chi_\lambda}= \frac{2|\xi|(1+\lambda^2 |\xi|^2)(\lambda^2-1)}{\lambda^2 (1+|\xi|^2)^2}	
	\leqslant c (\lambda^2-1).
\end{align}
Further
\begin{align}\label{Dchi /chi abs}\nonumber
	\frac{|\nabla_{S^2} \chi_\lambda|^2}{\chi_\lambda} &= \frac{4|\xi|^2 (1+\lambda^2 |\xi|^2)^2(\lambda^2-1)^2}{\lambda^4(1+|\xi|^2)^4}\cdot \frac{\lambda^2(1+|\xi|^2)^2}{(1+\lambda^2 |\xi|^2)^2}= \frac{4|\xi|^2(\lambda^2 -1)^2}{\lambda^2(1+|\xi|^2)^2}
	\\
	&\leqslant c(\lambda^2-1).
\end{align}
To estimate the Laplacian of $\chi_\lambda$ we calculate
\begin{align*}
	\sum_{i=1}^{2}\p_{ii}^2 \chi_\lambda(\xi)= \frac{8 (\lambda^2-1) (1+2\lambda^2 |\xi|^2)}{\lambda^2(1+|\xi|^2)^3}-\frac{24(\lambda^2-1)|\xi|^2(1+\lambda^2|\xi|^2)}{\lambda^2(1+|\xi|^2)^4}
\end{align*}
and 
\begin{align}\nonumber\label{Delta chi abs}
	|\Delta_{S^2} \chi_\lambda|&= \left|\frac{1}{\sqrt{\det g}} \p_i \left( g^{ij}\sqrt{\det g} \p_j \chi_\lambda\right)\right|	
	\\
	&=\bigg|\frac{2(\lambda^2-1)(1+2\lambda^2 |\xi|^2)}{\lambda^2 (1+|\xi|^2)}-\frac{6|\xi|^2(1+\lambda^2 |\xi|^2)(\lambda^2-1)}{\lambda^2 (1+|\xi|^2)^2}\bigg|
	\leqslant c(\lambda^2-1).
\end{align}
With this we show

\begin{lemma}\label{chiDeltaPsi small in L2 lemma}
	There exist $0<\varepsilon_0,\delta_0<1$ and a constant $C>0$ depending only on $\varepsilon_0$ and $\delta_0$  such that for every $0<\varepsilon<\varepsilon_0$, every $0<\delta<\delta_0$ and every critical point $v\in W^{2,2}(S^{2},S^{2})$ of $E_{\varepsilon,\lambda}$ satisfying (\ref{eq:delta-close2}) and (\ref{eq:alphalambd2}) we have 
	\begin{align*}
		||\sqrt{\chi_\lambda}\nabla^2 \psi||_{L^2(S^2)}+ \sqrt{\varepsilon}||\chi_\lambda \nabla^3 \psi||_{L^2(S^2)}\leqslant C (\delta+\varepsilon)\lambda,
	\end{align*}
	with $\psi=v-\Id$.	
\end{lemma}

\begin{proof}
	Note that  $||\psi||_{L^\infty(S^2)}\leqslant 2$.
	We start by estimating the mean value of $\psi$ with (\ref{eq:delta-close2}),  (\ref{EL psi proof}) and integration by parts. Note that $\int_{S^2}\Delta \left(\chi_\lambda (\Delta \psi + \Delta \Id)\right)=0$. 
	\begin{align}\label{mv psi 1st est}\nonumber
		\bigg|2\mint_{S^2} \psi dA_{S^2 }\bigg| &= \bigg|-\mint_{S^2} \bigg( \psi |\nabla \psi|^2 + 2 \psi\langle \nabla \psi, \nabla\Id\rangle + \Id |\nabla \psi|^2 + 2 \Id \langle \nabla \psi,\nabla \Id\rangle  \bigg)dA_{S^2}
		\\\nonumber
		&\qquad + \varepsilon \mint_{S^2} \bigg(\Delta \left(\chi_\lambda (\Delta \psi + \Delta \Id)\right)
		- (\psi+\Id)\chi_\lambda |\Delta \psi + \Delta \Id|^2\\\nonumber
		&\qquad\qquad +(\psi +\Id)  \Delta \left(\chi_\lambda |\nabla \psi +\nabla \Id|^2\right)
		\\\nonumber
		&\qquad\qquad + 2 (\psi+\Id) \Div \langle \chi_\lambda(\Delta \psi+\Delta \Id), \nabla \psi +\nabla \Id\rangle\bigg)dA_{S^2} \bigg|
		\\\nonumber
		&\leqslant c\int_{S^2} |\nabla \psi|^2dA_{S^2} + c\left(\int_{S^2} |\nabla \psi|^2dA_{S^2} \right)^\frac{1}{2} 
		\\\nonumber
		&\quad+ \varepsilon \bigg| \mint_{S^2}\bigg(- (\psi+\Id)\chi_\lambda |\Delta \psi + \Delta \Id|^2+(\psi +\Id)  \Delta \left(\chi_\lambda |\nabla \psi +\nabla \Id|^2\right)
		\\\nonumber
		&\qquad 
		+ 2 (\psi+\Id) \Div \langle \chi_\lambda(\Delta \psi+\Delta \Id), \nabla \psi +\nabla \Id\rangle\bigg) dA_{S^2} \bigg|
		\\\nonumber
		&\leqslant c\delta 	+ \varepsilon \bigg|\mint_{S^2}\bigg(- (\psi+\Id)\chi_\lambda |\Delta \psi + \Delta \Id|^2
		+(\Delta \psi +\Delta \Id)  \left(\chi_\lambda |\nabla \psi +\nabla \Id|^2\right)
		\\
		&\qquad 
		- 2 (\nabla\psi+\nabla \Id) \langle \chi_\lambda(\Delta \psi+\Delta \Id), \nabla \psi +\nabla \Id\rangle\bigg) dA_{S^2} \bigg|.
	\end{align}
	We estimate the remaining terms using Young's inequality, (\ref{eq:delta-close2}) and (\ref{eq:alphalambd2})
	\begin{align*}
		\varepsilon\bigg|\mint_{S^2} (\psi+\Id) \chi_\lambda |\Delta \psi +\Delta \Id|^2 dA_{S^2} \bigg|&\leqslant c\varepsilon \int_{S^2} \chi_\lambda (|\Delta \psi|^2 +1) dA_{S^2} 
		\\
		&\leqslant c(\delta^2+\varepsilon) + c\varepsilon(\lambda^2 -1)
		\\
		&\leqslant c(\delta +\varepsilon)
	\end{align*}
	and 
	\begin{align}\nonumber\label{eq 1 mean value proof}
		\varepsilon\bigg|\mint_{S^2} &(\Delta \psi +\Delta \Id) \left(\chi_\lambda |\nabla \psi +\nabla \Id|^2\right)dA_{S^2}  \bigg|
		\\\nonumber
		&\leqslant c\varepsilon \int_{S^2}\chi_\lambda (|\Delta \psi|+1)  \left(|\nabla \psi|^2  +1\right)dA_{S^2} 
		\\\nonumber
		&\leqslant c\varepsilon \int_{S^2} \chi_\lambda|\Delta \psi|^2 dA_{S^2} + c\varepsilon \int_{S^2}\chi_\lambda|\nabla \psi|^4 dA_{S^2} + c\varepsilon\int_{S^2} \chi_\lambda |\nabla \psi|^2dA_{S^2} + c\varepsilon \lambda^2
		\\
		&\leqslant c(\delta+\varepsilon)+c\varepsilon \int_{S^2} \chi_\lambda |\nabla \psi|^4dA_{S^2}  .	
	\end{align}
	For the last term we use the Sobolev embedding $W^{1,1}\hookrightarrow L^2(S^2) $, (\ref{eq:delta-close2}) and  (\ref{Dchi /chi abs})
	\begin{align*}
		\int_{S^2} \chi_\lambda |\nabla \psi|^4dA_{S^2}  &\leqslant c \left(\int_{S^2} \frac{|\nabla \chi_\lambda|}{\sqrt{\chi_\lambda}} |\nabla \psi|^2dA_{S^2}  \right)^2 + c \left(\int_{S^2} \chi_\lambda |\nabla ^2 \psi|^2dA_{S^2}  \right)\left( \int_{S^2} |\nabla \psi|^2dA_{S^2} \right)
		\\
		&\quad + \left(\int_{S^2} \sqrt{\chi_\lambda}|\nabla \psi|^2dA_{S^2} \right)^2
		\\
		&\leqslant c\delta^4 \lambda^2+ c\delta^2 \left(\int_{S^2} \chi_\lambda |\nabla ^2 \psi|^2 dA_{S^2} \right).
	\end{align*}
	To get an estimate on the full second derivative we integrate by parts and exchange derivatives.  By Lemma 2.1.2 in \cite{LammDiss} we have $|\nabla \Delta \psi-\Delta \nabla \psi|\leqslant c(|\nabla \psi|^3+ |\nabla \psi|)$ and therefore
	\begin{align*}
		\int_{S^2} \chi_\lambda |\nabla^2 \psi|^2dA_{S^2}  &\leqslant c\int_{S^2} |\nabla \chi_\lambda||\nabla \psi||\nabla ^2 \psi|dA_{S^2} +c \int_{S^2} \chi_\lambda |\Delta \psi|^2 dA_{S^2} 
		\\
		&\quad + c \int_{S^2}\chi_\lambda ( |\nabla \psi|^4 + |\nabla \psi|^2)dA_{S^2} 
		\\
		&\leqslant (c\delta^2 +\eta ) \int_{S^2} \chi_\lambda |\nabla^2 \psi|^2dA_{S^2}  + c_\eta \int_{S^2} \frac{|\nabla \chi_\lambda|^2}{\chi_\lambda}|\nabla \psi|^2 dA_{S^2} 
		\\
		&\quad + c \int_{S^2} \chi_\lambda|\Delta \psi|^2dA_{S^2}  + c\delta^2 \lambda^2.
	\end{align*}
	For $\delta, \eta >0$ small we absorb the first term to the left-hand side and with (\ref{Dchi /chi abs}) we have 
	\begin{align}\label{nabla2psi abs}
		\int_{S^2} \chi_\lambda |\nabla^2 \psi|^2 dA_{S^2} &\leqslant c\delta^2 \lambda^2 + c\int_{S^2} \chi_\lambda |\Delta \psi|^2dA_{S^2}
	\end{align}
	and thus 
	\begin{align}\label{nabla psi4 est}
		\int_{S^2} \chi_\lambda |\nabla \psi|^4 dA_{S^2} &\leqslant c\delta^4 \lambda^2 + c\delta^2 \int_{S^2} \chi_\lambda|\Delta \psi|^2dA_{S^2}. 
	\end{align}
	Going back to (\ref{eq 1 mean value proof}) and using the above estimates we get
	\begin{align*}
		\varepsilon\bigg|\mint_{S^2}(\Delta \psi +\Delta \Id)  \left(\chi_\lambda |\nabla \psi +\nabla \Id|^2\right)dA_{S^2} \bigg|&\leqslant c(\delta +\varepsilon).
	\end{align*}
	Analogously we estimate
	\begin{align*}
		\varepsilon\bigg|\mint_{S^2} - 2 (\nabla\psi+\nabla \Id) \langle \chi_\lambda(\Delta \psi+\Delta \Id), \nabla \psi +\nabla \Id\rangle dA_{S^2} \bigg|&\leqslant c(\delta +\varepsilon).
	\end{align*}
	Combining all these estimates in (\ref{mv psi 1st est}) we obtain 
	\begin{align}\label{mv psi}
		\bigg|\mint_{S^2} \psi dA_{S^2} \bigg| \leqslant c(\delta +\varepsilon).
	\end{align}
	Further we have with $W^{1,1}\hookrightarrow L^2(S^2)$, (\ref{Dchi /chi abs}), (\ref{nabla2psi abs}), (\ref{nabla psi4 est}) and Proposition \ref{p:clomob3}
	\begin{align*}
		\varepsilon\int_{S^2}\chi_\lambda^2 |\nabla \psi|^6 dA_{S^2} &\leqslant c\varepsilon \left(\int_{S^2} |\nabla \chi_\lambda| |\nabla\psi|^3  dA_{S^2}\right)^2 
		+ c\varepsilon\left(\int_{S^2} \chi_\lambda|\nabla^2 \psi| |\nabla \psi|^2  dA_{S^2}\right)^2
		\\
		&\quad + c\varepsilon \left( \int_{S^2} \chi_\lambda |\nabla \psi|^3 dA_{S^2}\right)^2
		\\
		&\leqslant c\varepsilon \left( \int_{S^2} \chi_\lambda |\nabla\psi|^4 dA_{S^2} \right)\left( \int_{S^2} \frac{|\nabla\chi_\lambda|^2}{\chi_\lambda} |\nabla \psi|^2 dA_{S^2}\right)
		\\
		&\quad + c\varepsilon \left(\int_{S^2} \chi_\lambda |\nabla^2 \psi|^2 dA_{S^2} \right)\left( \int_{S^2} \chi_\lambda |\nabla \psi|^4 dA_{S^2}\right)
		\\
		&\quad + c\varepsilon \left( \int_{S^2}\chi_\lambda |\nabla\psi|^4  dA_{S^2}\right)\left( \int_{S^2} \chi_\lambda |\nabla \psi|^2 dA_{S^2}\right)
		\\
		&\leqslant c\delta^4 \lambda^2 + c\delta^2 \int_{S^2} \chi_\lambda |\Delta \psi|^2dA_{S^2} 
	\end{align*}	
	and similarly
	\begin{align}\label{chi3nabla2psi4 est}\nonumber
		\varepsilon^2 \int_{S^2} \chi_\lambda^3 |\nabla^2 \psi|^4 dA_{S^2}  &\leqslant c\left(\varepsilon \int_{S^2} \sqrt{\chi_\lambda}|\nabla \chi_\lambda| |\nabla^2 \psi|^2  dA_{S^2} \right)^2 +c\left( \varepsilon \int_{S^2} \chi_\lambda^\frac{3}{2} |\nabla^3 \psi||\nabla^2 \psi|  dA_{S^2}  \right)^2 
		\\\nonumber
		&\quad + \left(\varepsilon \int_{S^2} \chi_\lambda^\frac{3}{2}|\nabla^2 \psi|^2 dA_{S^2}  \right)^2 
		\\\nonumber
		&\leqslant c\left(\varepsilon\int_{S^2} \frac{|\nabla \chi_\lambda|}{\sqrt{\chi_\lambda}}\chi_\lambda |\nabla^2 \psi|^2 dA_{S^2}   \right)^2 + c\left(\varepsilon \int_{S^2} \chi_\lambda^\frac{3}{2}|\nabla^2 \psi|^2  dA_{S^2}  \right)^2
		\\\nonumber
		&\quad + c\left(\varepsilon \int_{S^2} \chi_\lambda^2 |\nabla^3\psi|^2 dA_{S^2}  \right)\left( \varepsilon \int_{S^2} \chi_\lambda|\nabla^2 \psi|^2 dA_{S^2} \right)
		\\
		&\leqslant c\delta^4 \lambda^2 +c \delta^2 \varepsilon\int_{S^2} \chi_\lambda^2 |\nabla^3 \psi|^2 dA_{S^2} .
	\end{align}
	With this we can estimate the $L^2$-norm of $\chi_\lambda\nabla^3 \psi$.  As above we integrate by parts and exchange derivatives. By Lemma 2.1.2 in \cite{LammDiss} we have $|\nabla^2 \Delta  \psi - \Delta \nabla^2 \psi|\leqslant c( |\nabla^2 \psi||\nabla \psi| + |\nabla^2 \psi| + |\nabla \psi|^4 + |\nabla \psi|)$.  With Proposition \ref{p:clomob3} and the estimates above we get
	\begin{align*}
		\varepsilon\int_{S^2} \chi_\lambda^2 |\nabla^3 \psi|^2dA_{S^2} &\leqslant  c \varepsilon\int_{S^2} \chi_\lambda |\nabla \chi_\lambda| |\nabla^2 \psi||\nabla^3 \psi| dA_{S^2} + c\varepsilon \int_{S^2} \chi_\lambda^2 |\nabla \Delta \psi|^2 dA_{S^2}
		\\
		&\quad + c \varepsilon\int_{S^2} \chi_\lambda^2 |\nabla^2 \psi| \left( |\nabla^2 \psi||\nabla \psi|^2 + |\nabla^2 \psi|+ |\nabla \psi|^4 +|\nabla \psi| \right) dA_{S^2}
		\\
		&\leqslant \eta\varepsilon \int_{S^2}\chi_\lambda^2 |\nabla^3 \psi|^2 dA_{S^2} + c_\eta \varepsilon\int_{S^2} \chi_\lambda\frac{|\nabla\chi_\lambda|^2}{\chi_\lambda} |\nabla^2 \psi|^2 dA_{S^2}
		\\
		&\quad + c \varepsilon \int_{S^2} \chi_\lambda^2 |\nabla\Delta \psi|^2  dA_{S^2} + c_\eta \varepsilon^2 \int_{S^2} \chi_\lambda^3 |\nabla^2 \psi|^4 dA_{S^2} 
		+ c\int_{S^2} \chi_\lambda|\nabla \psi|^4 dA_{S^2} 
		\\
		&\quad + c\varepsilon \int_{S^2} \chi_\lambda^2 |\nabla^2 \psi|^2  dA_{S^2} 
		+ c\varepsilon \int_{S^2} \chi_\lambda^2 |\nabla \psi|^6 dA_{S^2}  + c\varepsilon \int_{S^2} \chi_\lambda^2 |\nabla \psi|^2  dA_{S^2} 
		\\
		&\quad + \eta \int_{S^2} \chi_\lambda |\nabla^2 \psi|^2  dA_{S^2} 
		\\
		&\leqslant c \delta^2 \lambda^2 + (c\delta^2 +\eta ) \int_{S^2} \chi_\lambda |\nabla^2 \psi|^2 dA_{S^2}  + (c_\eta \delta^2 +\eta) \varepsilon \int_{S^2}\chi_\lambda^2 |\nabla^3 \psi|^2 dA_{S^2} 
		\\
		&\quad + c\varepsilon\int_{S^2} \chi_\lambda^2 |\nabla \Delta \psi|^2dA_{S^2}. 
	\end{align*}
	and for $\delta,\eta>0$ small enough
	\begin{align}\label{D3psi est}
		\varepsilon\int_{S^2} \chi_\lambda^2|\nabla^3\psi|^2dA_{S^2}  \leqslant c\delta^2 \lambda^2 + (c \delta^2 +\eta ) \int_{S^2} \chi_\lambda |\nabla^2 \psi|^2 dA_{S^2}  + c\varepsilon\int_{S^2} \chi_\lambda^2 |\nabla \Delta \psi|^2dA_{S^2} .
	\end{align}
	Now we multiply (\ref{EL psi}) with $\chi_\lambda\Delta \psi$ and integrate over $S^2$. After rearranging we get
	\begin{align}\label{rhs eqn1}\nonumber
		&\int_{S^2} \left\langle (\Delta -\varepsilon \chi_\lambda \Delta^2)\psi,  \chi_\lambda \Delta \psi \right\rangle dA_{S^2}  
		\\\nonumber
		&= \int_{S^2} \bigg\langle \bigg(-\psi|\nabla \psi|^2 -2\psi\langle \nabla \psi , \nabla \Id\rangle -2\psi -\Id |\nabla \psi|^2 - 2 \Id \langle \nabla \psi , \nabla \Id \rangle\bigg)
		\\
		&\qquad + \varepsilon \bigg( \Psi_1- \chi_\lambda\Delta^2 \psi + \Psi_2+\Psi_3\bigg),\chi_\lambda \Delta \psi \bigg\rangle dA_{S^2}. 
	\end{align}
	We estimate the left-hand side further
	\begin{align*}
		&\int_{S^2}\left\langle (\Delta -\varepsilon \chi_\lambda \Delta^2)\psi,  \chi_\lambda \Delta \psi \right\rangle dA_{S^2} 
		\\
		& = \int_{S^2}\chi_\lambda|\Delta \psi|^2 dA_{S^2} + \varepsilon\int_{S^2} \chi_\lambda^2 |\nabla\Delta \psi|^2   
		dA_{S^2} +2\varepsilon \int_{S^2} \chi_\lambda \nabla \chi_\lambda\nabla \Delta \psi  \Delta \psi dA_{S^2} 
		\\
		&\geqslant \int_{S^2}\chi_\lambda|\Delta \psi|^2 dA_{S^2}  + \frac34 \varepsilon \int_{S^2} \chi_\lambda^2 |\nabla\Delta \psi|^2    dA_{S^2}
		-c\varepsilon   \int_{S^2} |\nabla \chi_\lambda|^2 |\Delta \psi|^2  dA_{S^2} . 
	\end{align*}
	Together with (\ref{eq:delta-close2}), (\ref{Dchi /chi abs}), (\ref{nabla2psi abs}) and (\ref{D3psi est}) we have 
	\begin{align}\nonumber\label{lhs est}
		&\int_{S^2}\chi_\lambda|\nabla^2 \psi|^2  dA_{S^2} + \varepsilon\int_{S^2} \chi_\lambda^2 |\nabla^3 \psi|^2  dA_{S^2} 
		\\\nonumber
		&\leqslant c\int_{S^2}\left\langle (\Delta -\varepsilon \chi_\lambda \Delta^2)\psi,  \chi_\lambda \Delta \psi \right\rangle  dA_{S^2}+ c\varepsilon  \int_{S^2}  \chi_\lambda\frac{|\nabla \chi_\lambda|^2}{\chi_\lambda} |\Delta \psi|^2   dA_{S^2}+ c\delta^2\lambda^2
		\\
		&\leqslant  c\int_{S^2}\left\langle (\Delta -\varepsilon \chi_\lambda \Delta^2)\psi,  \chi_\lambda \Delta \psi \right\rangle  dA_{S^2}+c\delta^2\lambda^2.
	\end{align}
	On the right-hand side of (\ref{rhs eqn1}) we have 
	\begin{align}\label{rhs est}\nonumber
		&\int_{S^2} \bigg\langle \bigg(-\psi|\nabla \psi|^2 -2\psi\langle \nabla \psi , \nabla \Id\rangle -2\psi -\Id |\nabla \psi|^2 - 2 \Id \langle \nabla \psi , \nabla \Id \rangle\bigg)
		\\\nonumber
		&\qquad + \varepsilon \bigg( (\Psi_1- \chi_\lambda\Delta^2 \psi) + \Psi_2+\Psi_3\bigg),\chi_\lambda \Delta \psi \bigg\rangle dA_{S^2}
		\\
		&=: I+II+III+IV.
	\end{align}
	We estimate each term separately. With Young's inequality, (\ref{eq:delta-close2}) and (\ref{nabla psi4 est}) 
	\begin{align*}
		I&=	\int_{S^2} \bigg\langle -\psi|\nabla \psi|^2 -2\psi\langle \nabla \psi , \nabla \Id\rangle -2\psi -\Id |\nabla \psi|^2 - 2 \Id \langle \nabla \psi , \nabla \Id \rangle,\chi_\lambda \Delta \psi \bigg\rangle dA_{S^2}
		\\
		&\leqslant \eta \int_{S^2} \chi_\lambda |\Delta \psi|^2  dA_{S^2}+ c_\eta \int_{S^2} \chi_\lambda\left( |\nabla \psi|^4 + |\nabla \psi|^2 \right) dA_{S^2}- 2 \int_{S^2} \langle\psi,\chi_\lambda\Delta \psi\rangle  dA_{S^2}
		\\
		&\leqslant (\eta + c_\eta\delta^2)\int_{S^2} \chi_\lambda |\Delta \psi|^2  dA_{S^2}+ c_\eta \delta^2 \lambda^2 - 2 \int_{S^2} \langle\psi,\chi_\lambda\Delta \psi\rangle dA_{S^2}.
	\end{align*}
	Let $\bar{\psi}= \mint_{S^2} \psi$ be the mean value of $\psi$. Integrating by parts, applying  the Poincar\'e inequality as well as (\ref{eq:delta-close2}), (\ref{Dchi abs}) and (\ref{mv psi}) yields
	\begin{align*}
		-2	\int_{S^2} \langle \psi,\chi_\lambda\Delta\psi\rangle dA_{S^2} &= -2\int_{S^2} \chi_\lambda\left[\left( \psi-\bar{\psi}\right) + \bar{\psi}\right] \Delta \psi dA_{S^2} 
		\\
		&\leqslant \eta \int_{S^2} \chi_\lambda |\Delta \psi|^2  dA_{S^2} + c_\eta \int_{S^2} \chi_\lambda|\psi -\bar{\psi}|^2  dA_{S^2} + 2 \bar{\psi}\int_{S^2} \nabla\chi_\lambda \nabla \psi dA_{S^2} 
		\\
		&\leqslant \eta \int_{S^2} \chi_\lambda |\Delta \psi|^2  dA_{S^2}  +c_\eta \lambda^2\int_{S^2} |\nabla \psi|^2  dA_{S^2} 
		\\
		&\quad + c(\lambda^2-1)|\bar{\psi}| \left( \int_{S^2} |\nabla \psi|^2  dA_{S^2} \right)^\frac{1}{2}
		\\
		&\leqslant \eta \int_{S^2} \chi_\lambda |\Delta \psi|^2  dA_{S^2} + c_\eta \delta(\delta+\varepsilon) \lambda^2.
	\end{align*}
	All in all  we have 
	\begin{align*}
		I&\leqslant  c\left(\eta+ c_\eta \delta^2\right) \int_{S^2} \chi_\lambda |\nabla^2 \psi|^2  dA_{S^2} + c_\eta \delta(\delta+\varepsilon) \lambda^2.
	\end{align*}
	To estimate the second term we use (\ref{eq:delta-close2}), (\ref{eq:alphalambd2}), (\ref{nabla2psi abs}), (\ref{nabla psi4 est}), (\ref{chi3nabla2psi4 est}) and (\ref{D3psi est}) 
	\begin{align*}
		II&= \varepsilon \int_{S^2}\bigg\langle \chi_\lambda\bigg[ -4\psi +(\psi+\Id)\bigg(4 \langle \nabla \Delta \psi, \nabla \psi \rangle  +4\langle \nabla \Delta \psi, \nabla \Id \rangle +|\Delta \psi|^2 + 2 |\nabla^2 \psi|^2 
		\\
		&\qquad+ 4 \langle\nabla^2 \psi, \nabla^2 \Id \rangle-4\Delta \psi \Id - 8 \langle \nabla \psi, \nabla \Id \rangle\bigg)\bigg], \chi_\lambda\Delta \psi\bigg\rangle dA_{S^2} 
		\\
		&\leqslant c \varepsilon\int_{S^2} \chi_\lambda^2 \bigg( |\nabla\Delta \psi||\Delta \psi||\nabla \psi| + |\nabla \Delta \psi||\Delta \psi|  + |\nabla^2 \psi|^3 + |\nabla^2 \psi|^2 + |\nabla \psi||\Delta \psi| +|\Delta \psi| \bigg) dA_{S^2} 
		\\
		&\leqslant \eta \varepsilon\int_{S^2} \chi_\lambda^2 |\nabla^3 \psi|^2 dA_{S^2} + \eta \int_{S^2} \chi_\lambda|\nabla^2 \psi|^2  dA_{S^2}  + c_\eta \varepsilon^2 \int_{S^2} \chi_\lambda^3|\nabla^2 \psi|^4  dA_{S^2} 
		\\
		&\quad  + c_\eta \varepsilon\int_{S^2}\chi_\lambda^2 |\nabla^2\psi|^2 dA_{S^2} + c_\eta \int_{S^2} \chi_\lambda|\nabla \psi|^4  dA_{S^2} + c\varepsilon\int_{S^2} \chi_\lambda^2 |\nabla \psi|^2 dA_{S^2} +c_\eta \varepsilon^2 \int_{S^2} \chi_\lambda^3 dA_{S^2}
		\\
		&\leqslant (\eta+ c_\eta\delta^2) \varepsilon\int_{S^2} \chi_\lambda^2 |\nabla^3 \psi|^2  dA_{S^2} + (\eta+c_\eta\delta^2 )\int_{S^2} \chi_\lambda|\nabla^2 \psi|^2 dA_{S^2} 
		+c(\delta+\varepsilon)^2\lambda^2. 
	\end{align*}
	Similarly we get for the next term with  (\ref{eq:delta-close2}), (\ref{eq:alphalambd2}), (\ref{Dchi /chi abs}), (\ref{nabla2psi abs}), (\ref{chi3nabla2psi4 est}) and (\ref{D3psi est}) 
	\begin{align*}
		III&= \varepsilon\int_{S^2} \bigg\langle\nabla_i\chi_\lambda\bigg[ 2\nabla_i\Delta \psi -4\nabla_i\Id + (\psi+\Id) \bigg( 4\langle \nabla_i\nabla\psi, \nabla\psi\rangle + 2 \langle \Delta \psi, \nabla_i\psi\rangle
		\\
		&\qquad + 4 \langle \nabla_i\nabla \psi, \nabla \Id \rangle  + 2 \langle \Delta \psi, \nabla_i\Id \rangle + 4 \langle \nabla \psi, \nabla_i\nabla\Id \rangle -4 \langle \nabla_i\psi, \Id \rangle \bigg)\bigg], \chi_\lambda\Delta \psi\bigg\rangle dA_{S^2} 
		\\
		&\leqslant c\varepsilon \int_{S^2} \chi_\lambda|\nabla \chi_\lambda|\bigg( |\nabla\Delta \psi||\Delta \psi| + |\nabla^2 \psi|^2 |\nabla\psi| + |\nabla^2 \psi|^2 + |\Delta \psi||\nabla \psi| + |\Delta \psi|\bigg)  dA_{S^2} 
		\\
		&\leqslant \eta \varepsilon\int_{S^2} \chi_\lambda^2 |\nabla \Delta \psi|^2  dA_{S^2} 
		+\eta \int_{S^2} \chi_\lambda|\nabla^2 \psi|^2 dA_{S^2}
		+ c_\eta\varepsilon \int_{S^2} \frac{|\nabla\chi_\lambda|^2}{\chi_\lambda}\chi_\lambda |\Delta \psi|^2  dA_{S^2} 
		\\
		&\quad + c\varepsilon^2 \int_{S^2} \chi_\lambda^3 |\nabla^2 \psi|^4  dA_{S^2} 
		+ c\int_{S^2} \frac{|\nabla\chi_\lambda|^2}{\chi_\lambda} |\nabla\psi|^2 dA_{S^2} 
		+ c\varepsilon\int_{S^2} \chi_\lambda|\nabla\chi_\lambda||\nabla^2 \psi|^2 dA_{S^2}  
		\\
		&\quad +  c\varepsilon \int_{S^2} \chi_\lambda|\nabla \chi_\lambda||\nabla\psi|^2  dA_{S^2} + c_\eta \varepsilon^2 \int_{S^2} \chi_\lambda|\nabla \chi_\lambda|^2 dA_{S^2}
		\\
		&\leqslant ( \eta+ c\delta^2 )\varepsilon\int_{S^2} \chi_\lambda^2 |\nabla^3 \psi|^2  dA_{S^2}+ \eta \int_{S^2} \chi_\lambda|\nabla^2 \psi|^2dA_{S^2}
		+c_\eta(\delta+\varepsilon)^2\lambda^2.
	\end{align*}	
	and finally  we have with (\ref{Delta chi abs})
	\begin{align*}
		IV&= \varepsilon \int_{S^2} \bigg\langle \Delta \chi_\lambda\bigg[ \Delta \psi + 2 \psi +(\psi+\Id )\bigg( |\nabla \psi|^2 + 2 \langle \nabla \psi, \nabla \Id \rangle \bigg)\bigg], \chi_\lambda\Delta \psi\bigg\rangle dA_{S^2}
		\\
		&\leqslant c\varepsilon \int_{S^2} \chi_\lambda|\Delta \chi_\lambda| \bigg( |\Delta \psi|^2+ |\Delta \psi||\nabla\psi|^2 + |\Delta \psi||\nabla\psi| + |\Delta \psi| \bigg) dA_{S^2}
		\\
		&\leqslant c\varepsilon\int_{S^2} \chi_\lambda|\Delta \chi_\lambda||\Delta \psi|^2 dA_{S^2} + c\varepsilon\int_{S^2} \chi_\lambda|\Delta \chi_\lambda||\nabla \psi|^4  dA_{S^2}+ c\varepsilon\int_{S^2} \chi_\lambda|\Delta \chi_\lambda||\nabla \psi|^2 dA_{S^2}
		\\
		&\quad + \eta \int_{S^2} \chi_\lambda|\Delta \psi|^2 dA_{S^2} + c_\eta \varepsilon^2\int_{S^2} \chi_\lambda |\Delta \chi_\lambda|^2 dA_{S^2}
		\\
		&\leqslant \eta \int_{S^2} \chi_\lambda|\Delta \psi|^2 dA_{S^2} +c(\delta+\varepsilon)^2\lambda^2.
	\end{align*}
	Now we put  (\ref{lhs est}), (\ref{rhs est}) and the above estimates together. Choosing $\delta$ and $\eta$ small enough so that we can absorb these terms to the left-hand side we arrive at  
	\begin{align*}
		\int_{S^2}\chi_\lambda|\nabla^2 \psi|^2 dA_{S^2} +\varepsilon\int_{S^2} \chi_\lambda^2 |\nabla^3 \psi|^2  dA_{S^2}&\leqslant c(\delta +\varepsilon)^2
		\lambda^2. 
	\end{align*}
	
\end{proof}

\begin{cor}\label{lambda2-1 small cor}
	There exist $\varepsilon_0>0$ and $\delta_0>0$, possibly smaller than those in Lemma \ref{chiDeltaPsi small in L2 lemma}, such that for every $0<\varepsilon\leqslant \varepsilon_0$, $0<\delta\leqslant \delta_0$ and every critical point $v\in W^{2,2}(S^2,S^2)$ of $E_{\varepsilon,\lambda}$ satisfying (\ref{eq:delta-close2}) and (\ref{eq:alphalambd2}),  we have 	
	\begin{align*}
		\lambda^2-1\leqslant c(\delta+\varepsilon).
	\end{align*}
	 Moreover, the following estimate holds
	\begin{align}\label{nabla3 psi est cor}
		||\psi||_{L^\infty(S^2)}+ ||\psi||_{W^{2,2}(S^2)}+\sqrt{\varepsilon }||\nabla^3\psi||_{L^2(S^2)}\leqslant c(\delta+\varepsilon).
	\end{align}
	
\end{cor}

\begin{proof}
	With (\ref{eq:lsqbd}) and Lemma \ref{chiDeltaPsi small in L2 lemma} we have 
	\begin{align*}
		C \ep (\lambda^2 - 1) &\leqslant \frac{d}{d\, \log \lambda} E_{\ep,\lambda}(\Id)-\frac{d}{d\, \log \lambda} E_{\ep,\lambda}(v) 
		\notag \\ 
		&= \ep \int_{S^2} \chi_\lambda z(\lambda \cdot) (|\Delta \Id|^2 - |\Delta v|^2) \, dA_{S^2} \notag \\ 
		&\leqslant  \sqrt{\ep} \| \sqrt{\chi_{\lambda}}\Delta (v-\Id)\|_{L^2(S^2)} 
		\sqrt{\ep} \big(\| \sqrt{\chi_{\lambda}}\Delta v\|_{L^2(S^2)} + \| \sqrt{\chi_{\lambda}}\Delta \Id\|_{L^2(S^2)}\big)
		\\
		&\leqslant c \varepsilon (\delta+\varepsilon) \lambda^2 
		= c\varepsilon (\delta+\varepsilon) (\lambda^2 -1) +c\varepsilon(\delta+\varepsilon).
	\end{align*}
	For $\varepsilon+\delta$ small enough
	\begin{align*}
		\lambda^2-1\leqslant c(\delta+\varepsilon)
	\end{align*}
	and $\lambda^2\leqslant 2$. But then 
	\begin{align*}
		\frac{1}{2}\leqslant \frac{1}{\lambda^2}\leqslant |\chi_\lambda|\leqslant \lambda^2 \leqslant 2.
	\end{align*}
	With this and Lemma \ref{chiDeltaPsi small in L2 lemma} we get 
	\begin{align*}
		\int_{S^2} |\nabla^2 \psi|^2 dA_{S^2}+ \varepsilon\int_{S^2} |\nabla^3 \psi|^2 dA_{S^2}\leqslant c(\delta +\varepsilon)^2.
	\end{align*}
	By the Sobolev embedding $W^{2,2}\hookrightarrow L^\infty(S^2)$, the Poincar\'e inequality and (\ref{mv psi}) it follows that
	\begin{align*}
		||\psi||_{L^\infty(S^2)}&\leqslant c||\psi ||_{W^{2,2}(S^2)} 
		\leqslant c(\delta+\varepsilon) + c||\psi-\bar{\psi}||_{L^2(S^2)}+ c|\bar{\psi}|
		\\
		&\leqslant c(\delta +\varepsilon) + c||\nabla \psi||_{L^2(S^2)}
		\leqslant c(\delta+\varepsilon).
	\end{align*}
	
\end{proof}

\begin{rmk}
	Note that 
	\begin{align}\label{chi-1 abs}
		|\chi_\lambda-1| &
		\leqslant \begin{cases}
			\lambda^2-1,\quad&\text{if }\chi_\lambda\geqslant 1\\ 1-\frac{1}{\lambda^2},\quad&\text{if }\chi_\lambda<1
		\end{cases}\bigg\}
		\leqslant \lambda^2-1
		\leqslant c(\delta+\varepsilon).
	\end{align}
	
\end{rmk}


\section{Optimal M\"obius transformation}

This section follows in parts chapter 6 in \cite{lammmalmic}. For a better comprehension of the arguments we repeat some of the calculations here.

So far our results suggest that there exists $M\in PSL(2,\mathbb{C})$ such that $u_M$ is close to the identity, however there is still some freedom in the choice of $M$. To show that $u_M=\Id$ and $\lambda$ is equal to one, we have to choose the optimal M\"obius transformation $M$ with corresponding eigenvalue $\lambda$ which minimizes $\| \nabla (u_M - \Id) \|_{L^2(S^2)}$. 
This can be done as in section 6 of \cite{lammmalmic}.

From now on we choose the optimal $M\in PSL(2,\mathbb{C})$ that minimizes $\| \nabla (u_M -\Id) \|_{L^2(S^2)}$. Let $v := u_M$ satisfy the assumptions of Lemma \ref{chiDeltaPsi small in L2 lemma} and Corollary \ref{lambda2-1 small cor}.  Our goal is to improve the bound in (\ref{nabla3 psi est cor}) to $\sqrt{\varepsilon}(\lambda^2-1)$.  In (\ref{EL psi}), a problematic term to estimate  is $\Id \langle \nabla \psi,\nabla \Id\rangle$, because it involves $\nabla\psi$ of order one. To eliminate this term we exploit that it is an element of the normal space at the identity.

By (\ref{nabla3 psi est cor}) $v$ converges pointwise to the identity map as $\delta$ and $\varepsilon$ tend to zero. Thus we can write $v$ in terms of the tangential component of $\psi$ at the identity
\[ 
v = \Id + \psi = \exp_{\Id} \hat{\psi} \quad (= \Id + \hat{\psi}  + O(|\hat{\psi}|^2)), 
\qquad \hat{\psi} \in T_{\Id} W^{3,2}(S^2,S^2).  
\] 
In the following we want to work with $\hat{\psi}$ instead of $\psi$. To do this we need formulas to express $\hat{\psi}$ in terms of $v$ and $\psi$. Let $\mathbf{x} = (x,y,z) \in S^2 \subset \R^3$, then 
\begin{align} 
	v(\mathbf{x}) &= \mathbf{x} \sqrt{1 - |\hat{\psi}(\mathbf{x})|^2} 
	\, + \, \hat{\psi}(\mathbf{x}), \qquad \hat{\psi}(\mathbf{x}) \cdot \mathbf{x} \equiv 0, \notag 
	\\ 
	\hat{\psi}(\mathbf{x}) &= \psi(\mathbf{x}) + \tfrac12 |\psi(\mathbf{x})|^2 \mathbf{x}\,, \qquad \qquad
	\psi(\mathbf{x}) = \hat{\psi}(\mathbf{x}) -
	\left( 1 - \sqrt{1 - |\hat{\psi}(\mathbf{x})|^2}\right) \mathbf{x},  \label{eq:psipsih} 
	\\ 
	|\hat{\psi}|^2 &= |\psi|^2(1 - \tfrac14|\psi|^2) \leqslant |\psi|^2 = 2(1 - \sqrt{1 - |\hat{\psi}|^2}). \notag 
\end{align} 
With this we get for the error terms of higher order
\begin{align}\label{fehler eqn}\begin{split}
		|\nabla \psi -\nabla \hat{\psi}|&=O(|\hat{\psi}| \,  |\nabla \hat{\psi}|) + O(|\hat{\psi}|^2)
		= O(|\psi| \,  |\nabla \psi|) + O(|\psi|^2),
		\\
		|\nabla^2 \psi -\nabla^2 \hat{\psi}| &=O(|\hat{\psi}| |\nabla^2 \hat{\psi}|) 
		+ O(|\nabla \hat{\psi}|^2) + O(|\hat{\psi}|^2)= O(|\psi| |\nabla^2 \psi|) 
		+ O(|\nabla \psi|^2) +O(|\psi|^2),
		\\
		|\nabla^3 \psi-\nabla^3\hat{\psi}|&= O(|\hat{\psi}||\nabla^3\hat{\psi}|)+ O(|\nabla^2\hat{\psi}|^2)+O(|\nabla\hat{\psi}|^2 )+O(|\hat{\psi}|^2)
		\\
		&= O(|\psi||\nabla^3\psi|)+O(|\nabla^2\psi|^2)+O(|\nabla\psi|^2)+O(|\psi|^2).
	\end{split}
\end{align} 
Let $\mathbf{x}^T$ be the orthogonal projection of $\mathbf{x}\in S^2$ onto the tangent space $T_\x S^2$. The tangential component of (\ref{EL psi}) is given by
\begin{align}\nonumber\label{EL tang}
	\bigg[  \varepsilon \chi_\lambda \lap^2 \psi-\lap \psi  -2\psi -4\varepsilon& \chi_\lambda\psi  \bigg]^T
	= \bigg[(\Id +\psi ) |\nabla \psi|^2 + 2(\Id +\psi) \langle \nabla \psi, \nabla \Id\rangle 
	\\\nonumber
	&\qquad\qquad -\varepsilon \bigg(\Psi_1(\psi,\Id)- \chi_\lambda(\lap^2\psi-4\psi)+ \Psi_2(\psi,\Id) +\Psi_3(\psi,\Id)\bigg)\bigg]^T
	\\\nonumber
	\Leftrightarrow
	-\varepsilon (\lap (\lap\hat{\psi} )^T)^T+(\lap \hat{\psi})^T &+2\hat{\psi} +4 \varepsilon \hat{\psi} 
	\\\nonumber
	=-2&\hat{\psi} \langle \nabla \hat{\psi}, \nabla\Id\rangle+ O(|\hat{\psi}||\nabla^2 \hat{\psi}|) + O(|\nabla \hat{\psi}|^2)+O(|\hat{\psi}||\nabla\hat{\psi}|) +O(|\hat{\psi}|^2)
	\\\nonumber
	+ \varepsilon& \bigg( (\chi_\lambda-1)(\lap (\lap\hat{\psi})^T)^T +\chi_\lambda (\lap(\lap  (\psi-\hat{\psi}))^T )^T +\chi_\lambda (\lap(\lap\psi)^N)^T \bigg)
	\\
	-\varepsilon& \bigg[ \Psi_1- \chi_\lambda\lap^2 \psi +4\chi_\lambda\psi +\Psi_2+  \Psi_3\bigg]^T.
\end{align} 
We shall denote by $J_\varepsilon$ the operator on the left hand side of \eqref{EL tang} 
and note that it can be written as 
\begin{align*}
 J_\varepsilon = \left( 1 -\varepsilon((\lap \cdot )^T -2)\right) ((\lap \cdot)^T +2). 
\end{align*}
Observe that $\left( 1 -\varepsilon((\lap \cdot )^T -2)\right)$ is a positive operator 
and therefore $J_\varepsilon$ has the same kernel as $J:=((\lap \cdot)^T +2)$. 

Let $\hat{\psi}= \hat{\psi}_0+ \hat{\psi}_1$, 
where $\hat{\psi}_0\in \operatorname{ker}J_\varepsilon$ and 
$\hat{\psi}_1\in (\operatorname{ker}J_\varepsilon)^\perp$ 
with respect to the inner product in $L^2$. 
$\lap^T\hat{\psi}_0=-2\hat{\psi}_0$ since $J_\varepsilon$ and $J$ have the same kernel. 
Further note that $J_\varepsilon$ is self-adjoint with respect to 
the inner product in $L^2(S^2, TS^2)$ and 
\begin{align*}
	\int_{S^2} \langle J_\varepsilon \hat{\psi}_1, (\lap \hat{\psi}_0)^T\rangle dA_{S^2} 
	= -2 \int_{S^2} \left\langle J_\varepsilon\hat{\psi}_1, \hat{\psi}_0\right\rangle dA_{S^2} 
	=-2 \int_{S^2} \left\langle \hat{\psi}_1, J_\varepsilon\hat{\psi}_0\right\rangle dA_{S^2} 
	=0.
\end{align*}
With this we get 
\begin{align}\label{J_eps 1st equ}\nonumber
	\int_{S^2} \left\langle J_\varepsilon \hat{\psi}, (\lap \hat{\psi})^T\right\rangle dA_{S^2}
	&= \int_{S^2} \left\langle J_\varepsilon \hat{\psi}_1, (\lap \hat{\psi}_1)^T\right\rangle dA_{S^2}
	\\\nonumber
	&= \varepsilon\int_{S^2} |\nabla (\lap \hat{\psi}_1)^T|^2 dA_{S^2}+ \int_{S^2} |(\lap \hat{\psi}_1)^T|^2 dA_{S^2}
	\\
	&\quad +(2+4\varepsilon) \int_{S^2} \langle \hat{\psi}_1,(\lap\hat{\psi}_1)^T\rangle dA_{S^2}.
\end{align}
We want to control the $\sqrt{\varepsilon} W^{3,2}$-norm of $\hat{ \psi}$ by the left-hand side of (\ref{J_eps 1st equ}). 
The first two terms on the right-hand side are positive, which leaves us with the last term. 
To get a control on this term we decompose the vectorfield $\hat{\psi}_1$ into eigenvectorfields of $(\lap \cdot )^T$. 
In \cite{lammmalmic}, these eigenvectorfields were obtained by comparing $(\lap \cdot )^T$ with 
the (rough) connection Laplacian  $\Delta_{TS^2}$ and with the Hodge Laplacian, 
but here we shall proceed more directly and derive them from gradients of the eigenfunctions of $\lap$ on $L^2(S^2)$. 

In the computations that follow, it will be convenient to 
denote the coordinates of $\x \in \R^3$ by $(x_1,x_2,x_3)$, 
the corresponding partial derivatives by $\del_1, \del_2,\del_3$, 
the $\R^3$ gradient by $\del = (\del_1, \del_2,\del_3)$ 
and the $\R^3$ laplacian $\del_1^1+\del_2^2+\del_3^2$ by $\del^2$. 
If $f$ is the restriction to $S^2$ of $F \in C^{\infty}(\R^3)$ 
and $\nabla f$ is its $S^2$-gradient then we have: 
\begin{equation} \label{eq:grad_radial} 
\nabla f(\x) = \del F(\x) - (\x \cdot \del F(\x)) \x 
\end{equation} 
where $\x \cdot \del F(\x) = \sum_{i=1}^3 x_i \del_iF(\x)$. 
The relation between the spherical Laplacian $\lap f$  and $\del^2 F$ is given by 
\begin{equation} \label{eq:laplace_radial} 
\lap f(\x) = \del^2 F(\x) - 2 \x \cdot \del F(\x) - \sum_{i,j=1}^3 x_ix_j \del_i \del_j F(\x). 
\end{equation} 
Let $F = P_k$, a homogeneous polynomial of degree $k$ that is harmonic on $\R^3$ 
and let $p_k$ be the restriction of $P_k$ to $S^2$. 
We see from \eqref{eq:laplace_radial} that $p_k$ is an eigenfunction of $\lap$ 
with eigenvalue equal to $-k(k+1)$. All eigenfunctions of $\lap$ on $S^2$ are of this form. 
We claim that 
\begin{equation}\label{eq:eigengrad} 
(\lap(\nabla p_k))^T = -k(k+1) \nabla p_k. 
\end{equation} 
To prove this claim, observe from \eqref{eq:grad_radial} that, as an $\R^3$-valued function, 
\[ 
\nabla p_k(\x) = \del P_k(\x) - k (P_k(\x)) \x. 
\] 
The components $\del_i P_k(\x)$ of $\del P_k(\x)$ are 
harmonic homogeneous polynomials of degree $(k-1)$ and therefore 
\[ 
\lap (\del P_k(\x)) = -k(k-1) \del P_k(\x). 
\] 
In the next calculation we let $e_1,e_2$ be an orthonormal basis of $T_{\x}S^2$ 
so that $D_{e_i}e_j(\x)=0$, where $D$ is the Levi-Civita covariant derivative on $TS^2$.  
We compute: 
\begin{align*} 
\lap ((P_k(\x)) \x) &= (\lap P_k(\x)) \x + 2 \sum_{i=1}^2e_i( P_k(\x)) e_i(\x) + (P_k(\x))(\lap \x) \\ 
&= -k(k+1) (P_k(\x)) \x + 2 \nabla p_k(\x) - 2 (P_k(\x))\x 
\end{align*} 
where we have used $\nabla P_k(\x) = (\del P_k(\x))^T = \nabla p_k(\x)$ 
and $\lap(\x) = -2\x$. We can now complete the proof of \eqref{eq:eigengrad}: 
\begin{align*} 
(\lap(\nabla p_k))^T &= (\lap (\del P_k(\x)))^T - k(\lap ((P_k(\x)) \x))^T \\ 
&= -k(k-1) \nabla p_k - 2k \nabla p_k \\ 
&= -k(k+1) \nabla p_k. 
\end{align*} 

To proceed further with the spectral analysis of $(\lap \cdot )^T$ 
we recall the Helmholtz-Hodge decomposition\footnote{A fuller discussion of 
the Helmholtz-Hodge decomposition of vector fields and their relation to knots in $\R^3$ 
can be found in \cite{CantarellaDeTurckGluck}} 
of vector fields on $S^2$. Let $\star$ denote 
the anticlockwise rotation by $90^{\circ}$ in $TS^2$. 
The \emph{curl} operator $\nabla \wedge$ is then defined by 
\[ 
\nabla \wedge \xi := - \nabla \cdot(\star \, \xi) \quad\text{for all $C^1$-vector fields $\xi$ on $S^2$}. 
\] 
It is easy to check that $\nabla \wedge (\nabla f) = 0 \ \forall \, f \in C^2(S^2)$. 
Since $S^2$ is simply connected, if $\nabla \wedge \xi = 0$ then there exists 
$f \colon S^2 \to \R$ (unique up to a constant) such that $\xi = \nabla f$. 
Similarly, if $\nabla \cdot \xi = 0$ then there exists 
$g \colon S^2 \to \R$ (unique up to a constant) such that $\xi = \star \nabla g$. 
We claim that if $\nabla \wedge \xi = 0$ and $\nabla \cdot \xi = 0$ then $\xi = 0$. 
This is because, writing $\xi$ as $\nabla f$ we get that 
\[ 
0 = \int_{S^2} f \nabla \cdot \xi dA_{S^2} = \int_{S^2} f (\lap f) dA_{S^2} = 
- \int_{S^2} |\nabla f|^2 dA_{S^2}. 
\] 
It follows that $\nabla f = 0$ and therefore $\xi = 0$, as claimed. 

Given any vector field $\xi$ on $S^2$ we define functions $f$ and $g$ 
to be the unique (up to constants) solutions of 
\[ 
\lap f = \nabla \cdot \xi, \qquad \lap g = \nabla \wedge \xi. 
\] 
Then, the divergence and curl of $(\xi - \nabla f - \star \nabla g)$ both vanish and therefore, 
\begin{equation} \label{eq:HHdecomp} 
\xi = \nabla f + \star \nabla g. 
\end{equation} 
\eqref{eq:HHdecomp} is called the Helmholtz-Hodge decomposition of $\xi$. 
It is an $L^2$-orthogonal decomposition because 
\[ 
\int_{S^2} \nabla f \cdot (\star \nabla g) dA_{S^2} = 0. 
\] 
We next want to show that if $p_k$ is as in \eqref{eq:eigengrad} we also have 
\begin{equation}\label{eq:star_eigengrad} 
(\lap(\star \nabla p_k))^T = -k(k+1) \star \nabla p_k. 
\end{equation} 
(It is worth pointing out that \eqref{eq:star_eigengrad} is immediate for the Hodge Laplacian.) 
Observe that if a vector field $\xi$ on $S^2$ is viewed as an $\R^3$-valued function that satisfies $\xi(\x) \cdot \x = 0$ for all $\x \in S^2$ then 
\[ 
\star \, \xi(\x) =  \x \times \xi(\x), 
\] 
where $\times$ is the cross-product on $\R^3$. Furthermore, if $\mathbf{a}$ is a fixed vector in $\R^3$ then 
\[ 
\mathbf{a} \cdot (\star \, \xi) = \xi \cdot (\mathbf{a} \times \x) 
\] 
and 
\[ 
\lap(\mathbf{a} \cdot (\star \, \xi)) = (\lap \xi) \cdot (\mathbf{a} \times \x) + 
2 \sum_{i=1}^2 e_i(\xi) \cdot (\mathbf{a} \times e_i) - 2\xi \cdot (\mathbf{a} \times \x). 
\] 
We now fix $\x \in S^2$ and let $\mathbf{a}$ run over the orthonormal basis $\x, e_1(\x), e_2(\x)$ 
oriented so that $\x \times e_1 = e_2$ to get 
\begin{align*} 
\lap(\star \, \xi) &= (-(\lap \xi - 2\xi) \cdot e_2 + 2 e_2(\xi) \cdot \x)e_1 \\ 
&\hphantom{=} +((\lap \xi - 2\xi) \cdot e_1 - 2 e_1(\xi) \cdot \x)e_2 \\ 
&\hphantom{=} + 2 (e_1(\xi)\cdot e_2 - e_2(\xi) \cdot e_1) \x. 
\end{align*} 
By differentiating the relation $\xi(\x) \cdot \x = 0$ we see that 
$e_1(\xi) \cdot \x + \xi \cdot e_1 = 0$ and $e_2(\xi) \cdot \x +\xi \cdot e_2 = 0$. Therefore 
\begin{equation}\label{eq:lapstar} 
(\lap(\star \, \xi))^T = -(\lap \xi \cdot e_2)e_1 + (\lap \xi \cdot e_1)e_2 = \star (\lap \xi)^T. 
\end{equation} 
\eqref{eq:star_eigengrad} follows immediately from \eqref{eq:eigengrad} and \eqref{eq:lapstar}. 

Let $\ph_0, \ph_1, \ph_2, \dotsc$ be a complete orthonormal set of eigenfunctions of $\lap$ on $S^2$. 
(We have remarked that these eigenfunctions are the restrictions of harmonic homogeneous polynomials on $\R^3$ to $S^2$.) 
Then, because of the decomposition \eqref{eq:HHdecomp}, the collection 
$\nabla \ph_1, \star \nabla \ph_1, \nabla \ph_2, \star \nabla \ph_2, \dotsc$ 
is a complete orthogonal set of eigenvectorfields of $(\lap \cdot)^T$. 
($\ph_0$ is constant and therefore $\nabla \ph_0 = 0$.) 
Furthermore, by \eqref{eq:eigengrad} and \eqref{eq:star_eigengrad}, 
the eigenvalues corresponding to $\nabla \ph_j$ and $\star \nabla \ph_j$ are the same as the eigenvalue 
(which is of the form $-k(k+1), \ k \in \N$) of $\ph_j$. 
In particular, the lowest eigenvalue of $-(\lap \cdot)^T$ is $2$ and it has multiplicity equal to $6$. 
The corresponding eigenvectorfields are $\nabla x_1, \ \nabla x_2, \ \nabla x_3$ (whose flows are dilatations) 
and 
$\star \nabla x_1, \ \star \nabla x_2, \ \star \nabla x_3$ (whose flows are rotations). 

We can now proceed with obtaining a lower bound for the right hand side of \eqref{J_eps 1st equ} 
by decomposing $W^{3,2}(S^2,TS^2)$ as $\oplus_{j=1}^\infty E_{\lambda_j}$, 
where $E_{\lambda_j}$ is the eigenspace of $(\lap \cdot)^T$ 
corresponding to the eigenvalue $\lambda_j = -j(j+1),~ j\in\N$. This decomposition enables us to express $\hat{\psi}$ in \eqref{J_eps 1st equ} as
\[ 
\hat{\psi} =\sum_{j=1}^\infty \hat{\psi}_{\lambda_j}, \quad \hat{\psi}_{\lambda_j}\in E_{\lambda_j}. 
\] 
As already remarked, the first eigenvalue is $\lambda_1=-2$ and 
the eigenvectorfield $\hat{\psi}_{\lambda_1}$ lies in the kernel of $J_\varepsilon$.
Therefore
\[ 
\hat{\psi}_1 =\sum_{j=2}^\infty \hat{\psi}_{\lambda_j}. 
\] 
Note that $\int_{S^2} \langle \hat{\psi}_{\lambda_i}, \hat{\psi}_{\lambda_j}\rangle=0$ if $i\neq j$. 
Then 
\begin{align}\label{eq:del1_psi} 
  2\int_{S^2}  \langle \hat{\psi}_1,(\lap\hat{\psi}_1)^T\rangle dA_{S^2} 
&= \sum_{j=2}^\infty\int_{S^2} 
  \left\langle \frac{2}{\lambda_j} (\lap\hat{\psi}_{\lambda_j})^T, (\lap \hat{\psi}_{\lambda_j})^T\right\rangle 
  dA_{S^2} \nonumber \\ 
&=\sum_{j=2}^\infty\frac{2}{\lambda_j}\int_{S^2} |(\lap\hat{\psi}_{\lambda_j})^T|^2 dA_{S^2}. 
\end{align}

Analogously we have 
\begin{align*}
  4\varepsilon\int_{S^2} 
  \langle \hat{\psi}_1, (\Delta \hat{\psi}_1)^T\rangle dA_{S^2}
&= -4\varepsilon\sum_{j=2}^{\infty} \int_{S^2} 
  \langle\nabla \hat{\psi}_{\lambda_j}, \nabla \hat{\psi}_{\lambda_j}\rangle 
  dA_{S^2} \\
&=  -4\varepsilon\sum_{j=2}^{\infty}\int_{S^2} \frac{1}{\lambda_j^2} 
  \langle\nabla(\lap \hat{\psi}_{\lambda_j})^T, \nabla(\lap \hat{\psi}_{\lambda_j})^T \rangle dA_{S^2} \\ 
&=  -4\varepsilon\sum_{j=2}^{\infty}\int_{S^2} \frac{1}{\lambda_j^2} 
  |\nabla(\lap \hat{\psi}_{\lambda_j})^T|^2 dA_{S^2} 
\end{align*} 
Inserting this in (\ref{J_eps 1st equ}) yields
\begin{align*}
  \int_{S^2} \left\langle J_\varepsilon \hat{\psi}, (\Delta \hat{\psi})^T\right\rangle  dA_{S^2} 
&= \varepsilon\int_{S^2}  |\nabla (\Delta \hat{\psi}_1)^T|^2  dA_{S^2} 
  + \int_{S^2} |(\Delta \hat{\psi}_1)^T|^2  dA_{S^2} \\ 
&\quad + \sum_{j=2}^\infty\frac{2}{\lambda_j}\int_{S^2} |(\lap\hat{\psi}_{\lambda_j})^T|^2 dA_{S^2} 
  -4\varepsilon\sum_{j=2}^{\infty}\int_{S^2} \frac{1}{\lambda_j^2} 
  |\nabla(\lap \hat{\psi}_{\lambda_j})^T|^2 dA_{S^2} \\ 
&= \sum_{j=2}^\infty\frac{\lambda_j + 2}{\lambda_j}\int_{S^2} |(\lap\hat{\psi}_{\lambda_j})^T|^2 dA_{S^2} 
  +\varepsilon\sum_{j=2}^{\infty}\int_{S^2} \frac{\lambda_j^2 - 4}{\lambda_j^2} 
  |\nabla(\lap \hat{\psi}_{\lambda_j})^T|^2 dA_{S^2}
\end{align*} 
Note that $\lambda_2=-6 > \lambda_3 > \lambda_4 \ \dots $. Therefore  
\begin{align*}
\frac{\lambda_j^2-4}{\lambda_j^2} \geqslant \frac89 \qquad\text{and}\qquad 
\frac{\lambda_j+2}{\lambda_j} \geqslant \frac23 \qquad \forall \, j \geqslant 2 
\end{align*} 
and so, 
\begin{align}\label{J_eps, Delta psi eqn 1}
  \int_{S^2} \left\langle J_\varepsilon \hat{\psi}, (\Delta \hat{\psi})^T\right\rangle dA_{S^2} 
&\geqslant \varepsilon \frac89 ||\nabla (\Delta\hat{\psi}_1)^T||_{L^2(S^2)}^2 
  + \frac23||(\Delta\hat{\psi_1})^T||_{L^2(S^2)}^2.
\end{align} 
We want to bound the full second and third derivative  of $\hat{\psi}_1$ in terms of 
$(\Delta \hat{\psi}_1)^T$ and $\nabla (\Delta \hat{\psi}_1)^T$. 
To do this we take a closer look at the normal part 
$(\Delta \hat{\psi_1})^N=(\Delta \hat{\psi}_1\cdot \x) \x$. 
Note that with $e_1,e_2$ the orthonormal basis for $T_\x S^2$ as before we have 
\begin{align}\label{normal part Delta}\nonumber
\Delta \hat{\psi}_1\cdot \x 
&= \sum_{i=1}^{2} e_i(e_i(\hat{\psi}_1))\cdot \x 
  = \sum_{i=1}^{2} e_i\left( e_i(\hat{\psi}_1)\cdot \x\right) - \left(e_i(\hat{\psi}_1)\cdot e_i\right)(\x)
  \\ \nonumber
&= \sum_{i=1}^{2} e_i\left(e_i\left(\hat{\psi}_1\cdot \x\right)\right) - 
  e_i\left(\hat{\psi}_1\cdot e_i\right)(\x) - \left(e_i(\hat{\psi}_1)\cdot e_i\right)(\x)
  \\
&= -2 \sum_{i=1}^{2} (e_i(\hat{\psi}_1)\cdot e_i)(\x)
  = -2 \operatorname{div}\hat{\psi}_1,
\end{align} 
where we used that $\hat{\psi}_1$ is tangential in the second line and 
$\hat{\psi}_1 \cdot e_i(e_i)=\hat{\psi} \cdot D_{e_i}e_i=0$ in the last line.
Then 
\begin{align*} 
(\Delta \hat{\psi}_1)^T = 
  \Delta \hat{\psi}_1 + 2\operatorname{div}\hat{\psi}_1\cdot \x 
\qquad\text{and}\qquad 
\varepsilon \nabla(\Delta \hat{\psi}_1)^T = 
  \varepsilon \nabla (\Delta \hat{\psi}_1) + 2\varepsilon \nabla (\operatorname{div}\hat{\psi}_1\cdot \x)
\end{align*}
and therefore 
\begin{align*} 
||\Delta \hat{\psi}_1||_{L^2(S^2)}
&\leqslant ||(\Delta \hat{\psi}_1)^T||_{L^2(S^2)} + c||\nabla \hat{\psi}_1||_{L^2(S^2)} 
\\ 
\sqrt{\varepsilon}||\nabla\Delta\hat{\psi}_1||_{L^2(S^2)}
&\leqslant \sqrt{\varepsilon}||\nabla(\Delta  \hat{\psi}_1)^T||_{L^2(S^2)} 
  + c\sqrt{\varepsilon}\left(||\nabla^2 \hat{\psi}_1||_{L^2(S^2)} + ||\nabla \hat{\psi}_1||_{L^2(S^2)}\right). 
\end{align*}
Integrating by parts we get
\[ 
\int_{S^2} |\nabla\hat{\psi}_1|^2dA_{S^2} = 
  -\int_{S^2} \hat{\psi}_1 \Delta \hat{\psi}_1 dA_{S^2} = 
  \sum_{j=2}^\infty\frac{1}{(-\lambda_j)}\int_{S^2} |(\lap\hat{\psi}_{\lambda_j})^T|^2 dA_{S^2} 
\leqslant \frac16 ||\Delta \hat{\psi}_1||_{L^2(S^2)}^2 
\] 
where we have used \eqref{eq:del1_psi} and $-\lambda_j \geqslant 6 \ \forall \, j \geqslant 2$. 
We also have 
\[ 
||\hat{\psi}_1||^2_{L^2(S^2)} = \sum_{j=2}^\infty\int_{S^2} 
  \frac{1}{\lambda_j^2} \left\langle (\lap\hat{\psi}_{\lambda_j})^T, (\lap \hat{\psi}_{\lambda_j})^T\right\rangle 
  dA_{S^2} \leqslant \frac{1}{36} ||(\lap \hat{\psi}_1)^T||^2_{L^2(S^2)}. 
\] 
Therefore, going back to (\ref{J_eps, Delta psi eqn 1}) we see that 
\begin{align*}
\sqrt{\varepsilon}||\nabla\Delta\hat{\psi}_1||_{L^2(S^2)} 
&+ ||\Delta \hat{\psi}_1||_{L^2(S^2)} + ||\nabla \hat{\psi}_1||_{L^2(S^2)} + ||\hat{\psi}_1||_{L^2(S^2)} \\ 
&\leqslant c\left(\int_{S^2} 
  \langle J_\varepsilon \hat{\psi}, (\Delta \hat{\psi}_1)^T\rangle dA_{S^2} \right)^\frac{1}{2} 
  + c\sqrt{\varepsilon}||\nabla^2 \hat{\psi}_1||_{L^2(S^2)} .
\end{align*}	
To get an estimate for the full second and third derivative we integrate by parts and 
exchange derivatives as in the proof of Lemma \ref{chiDeltaPsi small in L2 lemma}.
All in all we arrive at
\begin{align}\label{psi 1 J esp eqn}
\sqrt{\varepsilon} ||\nabla^3 \hat{\psi}_1||_{L^2(S^2)} + ||\hat{\psi}_1||_{W^{2,2}(S^2)} 
  \leqslant c \left(\int_{S^2} 
  \langle J_\varepsilon \hat{\psi}, (\Delta \hat{\psi})^T\rangle dA_{S^2}\right)^\frac{1}{2}. 
\end{align}
We have seen that the kernel of $J$, and therefore the kernel of $J_\varepsilon$, 
is finite (in fact $6$) dimensional and all norms on it are equivalent. We estimate
\begin{align*}
\sqrt{\varepsilon}||\nabla^3\hat{\psi}_{0}||_{L^2(S^2)} + ||\hat{\psi}_{0}||_{W^{2,2}(S^2)} 
  \leqslant  c ||\hat{\psi}_{0}||_{L^2(S^2)}.
\end{align*}
Together with (\ref{psi 1 J esp eqn})
\begin{align}\label{D3 nabla est 1}
\sqrt{\varepsilon}||\nabla^3 \hat{\psi}||_{L^2(S^2)} + ||\hat{\psi}||_{W^{2,2}(S^2)} 
  \leqslant \left(\int_{S^2} 
    \langle J_\varepsilon\hat{\psi},(\Delta \hat{\psi})^T\rangle dA_{S^2}\right)^\frac{1}{2} 
  + ||\hat{\psi}_0||_{L^2(S^2)}.
\end{align} 

At this point we use the fact that $M \in PSL(2,\C)$ has been selected 
so as to minimize $\| \nabla (u_M -\Id) \|_{L^2(S^2)}$. This implies that 
\begin{equation}\label{eq:Zorth} 
-\int_{S^2}\nabla v \cdot \nabla \xi \, dA_{S^2} 
+ \int_{S^2}\nabla Id \cdot \nabla \xi \, dA_{S^2} = 0 
\quad \forall \, \xi \in Z, 
\end{equation} 
where $Z$ is the tangent space of the M\"{o}bius group at the identity and 
$v$ is related to $\hat{\psi}$ by \eqref{eq:psipsih}. But $Z = \ker{J} = \ker{J_{\varepsilon}}$ and, 
as shown in \S6 of \cite{lammmalmic}, this makes it possible to use \eqref{eq:Zorth} 
to estimate $\hat{\psi}_0$ as follows
\begin{align*}
  ||\hat{\psi}_0||_{L^2(S^2)} \leqslant c||\hat{\psi}||_{L^\infty(S^2)}^2 
  \leqslant c ||\hat{ \psi}||_{L^\infty(S^2)} ||\hat{ \psi}||_{W^{2,2}(S^2)}.
\end{align*} 
Going back to (\ref{D3 nabla est 1}) and choosing $\delta+\varepsilon$ in Corollary \ref{lambda2-1 small cor}  small enough gives
\begin{align}\label{J eps abs below}
	\sqrt{\varepsilon}||\nabla^3 \hat{\psi}||_{L^2(S^2)}  + ||\hat{\psi}||_{W^{2,2}(S^2)}
	&\leqslant  \left(\int_{S^2}\langle J_\varepsilon\hat{\psi}, (\Delta \hat{\psi})^T\rangle dA_{S^2}\right)^\frac{1}{2} .
\end{align}
Now we estimate the right hand side further. By (\ref{EL tang})
\begin{align}\nonumber\label{psi-hat psi inegr}
	\int_{S^2}& \left\langle J_\varepsilon\hat{\psi},(\Delta \hat{\psi})^T\right\rangle dA_{S^2}
	\\\nonumber
	&=\int_{S^2} \bigg\langle-2\hat{\psi} \langle \nabla \hat{\psi}, \nabla\Id\rangle+ O(|\hat{\psi}||\nabla^2 \hat{\psi}|) + O(|\nabla \hat{\psi}|^2)+O(|\hat{\psi}||\nabla\hat{\psi}|) +O(|\hat{\psi}|^2)
	\\\nonumber
	&\qquad+ \varepsilon \bigg( (\chi_\lambda-1)(\Delta (\Delta\hat{\psi})^T)^T +\chi_\lambda (\Delta(\Delta  (\psi-\hat{\psi}))^T )^T +\chi_\lambda (\Delta(\Delta\psi)^N)^T \bigg)
	\\
	&\qquad -\varepsilon \bigg[ \Psi_1- \chi_\lambda\Delta^2 \psi +4\chi_\lambda\psi +\Psi_2+  \Psi_3\bigg]^T, (\Delta \hat{\psi})^T\bigg\rangle dA_{S^2}.
\end{align}
We estimate each part separately. For the first part we use  $W^{1,1}\hookrightarrow L^2(S^2)$, H\"older's inequality 
and Corollary \ref{lambda2-1 small cor} 
\begin{align*}
	&\int_{S^2} \left\langle 2\hat{\psi}\langle \nabla \hat{\psi}, \nabla \Id\rangle +O(|\hat{\psi}||\nabla^2 \hat{\psi}|) + O(|\nabla \hat{\psi}|^2)+O(|\hat{\psi}||\nabla\hat{\psi}|) +O(|\hat{\psi}|^2) , (\Delta \hat{\psi})^T\right\rangle dA_{S^2}
	\\
	&\leqslant \int_{S^2} |\nabla^2 \hat{\psi}|\bigg( |\hat{\psi}||\nabla^2 \hat{\psi}|+ |\nabla\hat{\psi}|^2 + |\hat{\psi}|^2 \bigg) dA_{S^2}
	\\
	&\leqslant (\eta +c(\delta +\varepsilon))\int_{S^2} |\nabla^2 \hat{\psi}|^2 dA_{S^2} + c_\eta\int_{S^2} \left(|\nabla\hat{\psi}|^4 + |\hat{\psi}|^4 \right) dA_{S^2}
	\\
	&\leqslant (\eta +c(\delta +\varepsilon))\int_{S^2} |\nabla^2 \hat{\psi}|^2 dA_{S^2}  + c_\eta \left( \int_{S^2} |\nabla^2 \hat{\psi}|^2 dA_{S^2}\right)\left( \int_{S^2} |\nabla \hat{\psi}|^2 dA_{S^2}\right) 
	\\
	&\quad +c_\eta\left(\int_{S^2} |\nabla \hat{\psi} |^2 dA_{S^2}\right)^2 + c_\eta \left(\int_{S^2} |\nabla \hat{\psi}|^2 dA_{S^2}\right)\left(\int_{S^2}|\hat{ \psi}|^2 dA_{S^2}\right) + c_\eta \left(\int_{S^2}|\hat{ \psi}|^2 dA_{S^2}\right)^2
	\\
	&\leqslant \left(\eta+c(\delta +\varepsilon)+c_\eta(\delta +\varepsilon)^2\right) || \hat{\psi}||_{W^{2,2}(S^2)}^2.
\end{align*}
Note that this is where the estimate fails if we include the term $\Id \langle \nabla \psi, \nabla \Id\rangle$ and this is the reason why we consider only tangential terms.

In the second part we use integration by parts, $W^{1,1}\hookrightarrow L^2(S^2)$, (\ref{Dchi abs}), (\ref{chi-1 abs}), (\ref{fehler eqn}) and Corollary \ref{lambda2-1 small cor}
\begin{align*}
	\varepsilon&\int_{S^2} \bigg\langle  (\chi_\lambda-1)(\Delta (\Delta\hat{\psi})^T)^T +\chi_\lambda (\Delta(\Delta  (\psi-\hat{\psi}))^T )^T +\chi_\lambda (\Delta(\Delta\psi)^N)^T, (\Delta \hat{\psi})^T\bigg\rangle dA_{S^2} 
	\\
	&=-\varepsilon\int_{S^2}  (\chi_\lambda-1) |\nabla(\Delta \hat{\psi})^T|^2 +\chi_\lambda \left\langle \nabla\left(\Delta(\psi-\hat{\psi})\right)^T,\nabla(\Delta\hat{\psi})^T\right\rangle dA_{S^2} 
	\\
	&\quad  - \varepsilon\int_{S^2} \chi_\lambda\left\langle\nabla(\psi |\nabla \psi|^2), \nabla (\Delta \hat{\psi})^T\right\rangle dA_{S^2} 
	\\
	&\quad - \varepsilon\int_{S^2} \nabla \chi_\lambda \left\langle \nabla (\Delta \hat{ \psi})^T, (\Delta \hat{ \psi})^T \right\rangle + \nabla \chi_\lambda\left\langle \nabla\left(\Delta(\psi-\hat{\psi})\right)^T,(\Delta\hat{\psi})^T\right\rangle dA_{S^2} 
	\\
	&\quad -\varepsilon \int_{S^2} \nabla \chi_\lambda \left\langle\nabla(\psi |\nabla \psi|^2), (\Delta \hat{\psi})^T\right\rangle dA_{S^2} 
	\\
	&\leqslant \varepsilon((\lambda^2-1) + \eta+c_\eta(\delta+\varepsilon))\int_{S^2}|\nabla^3 \hat{\psi}|^2  dA_{S^2} 
	\\
	&\quad + (c_\eta + (\lambda^2-1))\varepsilon \int_{S^2} |\nabla^2 \hat{\psi}|^4 +|\nabla \hat{\psi}|^4 + |\hat{\psi}|^4 dA_{S^2} 
	+ c\varepsilon(\lambda^2-1) \int_{S^2} |\nabla^2 \hat{\psi}|^2 dA_{S^2} 
	\\
	&\quad + (c_\eta + (\lambda^2-1))\varepsilon \int_{S^2} |\nabla \psi|^6 +|\nabla^2 \psi|^2 |\nabla\psi|^2 |\psi|^2dA_{S^2} 
	\\
	&\leqslant \varepsilon(\eta+c_\eta(\delta+\varepsilon))\int_{S^2}|\nabla^3 \hat{\psi}|^2 dA_{S^2}  + c_\eta \varepsilon(\delta +\varepsilon)|| \hat{\psi}||_{W^{2,2}(S^2)}^2.
\end{align*}
In the same way  we estimate
\begin{align*}
	\varepsilon \int_{S^2}&\left\langle  (\Psi_1- \chi_\lambda\Delta^2 \psi+4\chi_\lambda\psi)^T, (\Delta \hat{\psi})^T\right\rangle dA_{S^2} 
	\\
	&= \varepsilon \int_{S^2}\bigg\langle \chi_\lambda \bigg[ 4\langle \nabla \Delta \psi, \nabla \psi\rangle +2 |\nabla^2 \psi|^2 +|\Delta \psi|^2 +6\langle \nabla \Delta \psi, \nabla \Id\rangle
	\\
	&\qquad \qquad +4 \langle \nabla_i \nabla \psi , \nabla_i \nabla \Id \rangle -4 \langle \Delta \psi, \Delta \Id \rangle -8 \langle \nabla \psi, \nabla  \Id\rangle \bigg]^T, (\Delta \hat{\psi})^T\bigg\rangle   dA_{S^2} 
	\\
	&\leqslant  \varepsilon (\eta+c_\eta(\delta+\varepsilon))\int_{S^2} |\nabla^3 \hat{\psi}|^2 dA_{S^2} 
	+   c_\eta \varepsilon|| \hat{\psi}||^2_{W^{2,2}(S^2)}.
\end{align*}
For the third term we use (\ref{Dchi abs}),  $W^{1,1}\hookrightarrow L^2(S^2)$ and Corollary \ref{lambda2-1 small cor}
\begin{align*}
	\varepsilon \int_{S^2} \left\langle \Psi_2^T, (\Delta \hat{\psi})^T\right\rangle dA_{S^2} 
	&= \varepsilon \int_{S^2} \bigg\langle  \bigg[\nabla_i\chi_\lambda (\Id +\psi)\bigg( 4 \langle \nabla_i \nabla \psi, \nabla \psi \rangle + 4 \langle \nabla_i \nabla \psi ,\nabla\Id\rangle + 2 \langle \Delta \psi, \nabla_i \psi \rangle 
	\\
	&\qquad \qquad + 2 \langle \Delta \psi, \nabla_i \Id \rangle + 4 \langle \nabla \psi, \nabla_i\nabla \Id \rangle -4\langle \nabla_i \psi,  \Id \rangle \bigg)
	\\
	&\qquad \qquad+2 \nabla_i \chi_\lambda \left( \nabla_i \Delta \psi-2\nabla_i\Id\right) \bigg]^T, (\Delta \hat{\psi})^T\bigg\rangle dA_{S^2} 
	\\
	&\leqslant c\varepsilon \int_{S^2} |\nabla \chi_\lambda||\Delta \hat{\psi}| \left(|\nabla^3 \psi|+ |\nabla^2 \psi||\nabla \psi| + |\nabla^2 \psi| + |\nabla\psi|+ 1\right)dA_{S^2}
	\\
	&\leqslant c_\eta\varepsilon^2(\lambda^2-1)^2 \int_{S^2}\left(|\nabla^3 \psi|^2+|\nabla^2 \psi|^2 |\nabla\psi|^2 + |\nabla^2\psi|^2 + |\nabla \psi|^2 \right)  dA_{S^2} 
	\\
	&\quad + \eta\int_{S^2} |\nabla^2\hat{ \psi}|^2 dA_{S^2} + c\varepsilon(\lambda^2-1)\int_{S^2}|\nabla^2 \hat{\psi}|  dA_{S^2} 
	\\
	&\leqslant c_\eta (\delta +\varepsilon)\varepsilon(\lambda^2-1)^2 + 2\eta \int_{S^2} |\nabla^2 \hat{\psi}|^2  dA_{S^2} 
\end{align*}	
and with (\ref{Delta chi abs})
\begin{align*}
	\varepsilon \int_{S^2} &\left\langle  \Psi_3^T, (\Delta \hat{\psi})^T\right\rangle  dA_{S^2} 
	\\
	&= \varepsilon \int_{S^2} \bigg\langle \Delta \chi_\lambda \left( (\Id +\psi) (|\nabla \psi|^2 + 2 \langle \nabla \psi, \nabla \Id \rangle)+ \Delta \psi+2\psi\right)^T,(\Delta \hat{\psi})^T\bigg\rangle dA_{S^2} 
	\\
	&\leqslant c_\eta \varepsilon^2 (\lambda^2-1)^2 \int_{S^2} \left(|\nabla^2 \psi|^2 + |\nabla\psi|^4 + |\nabla \psi|^2 + |\psi|^2\right) dA_{S^2}
	\\
	&\quad   + \eta \int_{S^2}|\nabla^2 \hat{\psi}|^2  dA_{S^2} 
	\\
	&\leqslant c_\eta (\delta +\varepsilon)^2 \varepsilon^2 (\lambda^2-1)^2 + \eta \int_{S^2}|\nabla^2 \hat{\psi}|^2 dA_{S^2} . 
\end{align*}	
Thus we have in (\ref{psi-hat psi inegr})
\begin{align}\nonumber
	\int_{S^2} &\left\langle J_\varepsilon \hat{ \psi} , (\Delta \hat{\psi})^T\right\rangle dA_{S^2} 
	\leqslant c(\eta+\varepsilon+\delta)\left( \varepsilon ||\nabla^3 \hat{\psi}||_{L^2(S^2)}^2+|| \hat{\psi}||^2_{W^{2,2}(S^2)}\right)
	+ c(\delta +\varepsilon)\varepsilon(\lambda^2-1)^2.
\end{align}	
We apply this to (\ref{J eps abs below}) and choose   $\eta,\varepsilon,\delta>0$ small enough so that we can absorb the higher order terms to the left-hand side.
\begin{align}\label{hat psi w22 est}
	\sqrt{\varepsilon}||\nabla^3 \hat{\psi}||_{L^2(S^2)} +  ||\hat{\psi}||_{W^{2,2}(S^2)} &\leqslant \left( \int_{S^2} \langle J_\varepsilon \hat{\psi}, (\Delta \hat{\psi})^T\rangle  dA_{S^2} \right)^\frac{1}{2}
	\leqslant c(\delta+\varepsilon)^\frac{1}{2}\sqrt{\varepsilon} (\lambda^2-1). 
\end{align}
To get the same bound for $\psi$ we estimate with (\ref{fehler eqn}), $W^{1,1}\hookrightarrow L^2(S^2)$ and Corollary \ref{lambda2-1 small cor}
\begin{align*}
	\varepsilon\int_{S^2} |\nabla^3 \psi|^2 dA_{S^2}&\leqslant \varepsilon\int_{S^2} |\nabla^3 \hat{\psi}|^2 dA_{S^2}+ c\varepsilon \int_{S^2} \left(|\psi|^2 |\nabla^3\psi|^2 + |\nabla^2\psi|^4 + |\nabla\psi|^4+ |\psi|^4\right) dA_{S^2}
	\\
	&\leqslant \varepsilon\int_{S^2} |\nabla^3 \hat{\psi}|^2 dA_{S^2}+c\varepsilon(\delta +\varepsilon)^2 \int_{S^2}\left( |\nabla^3 \psi|^2 + |\nabla^2 \psi|^2 + |\nabla \psi|^2 + |\psi|^2\right) dA_{S^2}
\end{align*}
Analogously we get 
\begin{align*}
	\int_{S^2} |\nabla^2 \psi|^2 dA_{S^2}&\leqslant \int_{S^2} |\nabla^2 \hat{\psi}|^2 dA_{S^2}+ c\int_{S^2} \left(|\psi|^2 |\nabla^2 \psi|^2 + |\nabla \psi|^4 + |\psi|^4 \right)dA_{S^2}
	\\
	&\leqslant \int_{S^2} |\nabla^2 \hat{\psi}|^2 dA_{S^2}+ c(\delta+\varepsilon)^2\int_{S^2} \left( |\nabla^2 \psi|^2 + |\nabla \psi|^2 + |\psi|^2 \right)dA_{S^2}
\end{align*}
and 
\begin{align*}
	\int_{S^2} |\nabla \psi|dA_{S^2} &\leqslant \int_{S^2} |\nabla \hat{\psi}|^2 dA_{S^2}+ c(\delta+\varepsilon)^2 \int_{S^2} \left(|\nabla\psi|^2 + |\psi|^2\right) dA_{S^2},
	\\
	\int_{S^2} |\psi|^2dA_{S^2} &\leqslant \int_{S^2} |\hat{\psi}|^2 dA_{S^2}+ c(\delta +\varepsilon)^2\int_{S^2}|\psi|^2 dA_{S^2}.
\end{align*}
Putting everything together and absorbing terms yields 
\begin{align}
	\label{psi w22 est}
	\sqrt{\varepsilon}||\nabla^3 \psi||_{L^2(S^2)} +  ||\psi||_{W^{2,2}(S^2)} &\leqslant \left( \int_{S^2} \langle J_\varepsilon \hat{\psi}, (\Delta \hat{\psi})^T\rangle  dA_{S^2} \right)^\frac{1}{2}
	\leqslant c(\delta+\varepsilon)^\frac{1}{2}\sqrt{\varepsilon} (\lambda^2-1). 
\end{align}
With this we show the following

\begin{thm}\label{main thm 1}
	There exist $\delta>0$ and $\bar{\varepsilon}>0$ small such that the only critical points $u_\varepsilon$ of $E_\varepsilon$ of degree $\pm 1$ with $E_\varepsilon(u_\varepsilon)\leqslant 4\pi(1+2\varepsilon) +\delta$ and $\varepsilon\leqslant \bar{\varepsilon}$ are maps of the form $u^R(x)= Rx,~ R\in O(3)$.
\end{thm}

\begin{proof}
	Since a map of degree $-1$ only differs from a map of degree $1$ by a reflection, we can  assume without loss of generality that $u_\varepsilon$ is a critical point of degree one. Let $M$ be the M\"obius transformation that minimizes $||(u_\varepsilon)_M -\Id||_{L^2(S^2)}$ and let $v=(u_\varepsilon)_M$.  We use (\ref{psi w22 est}) to estimate (\ref{eq:lsqbd}) further 
	\begin{align*}
		C\varepsilon (\lambda^2 -1) &\leqslant \frac{ d}{d\log \lambda}E_{\varepsilon, \lambda}( \Id)- \frac{d}{d\log \lambda}E_{\varepsilon, \lambda}(v)
		\\
		&\leqslant 2\sqrt{\varepsilon}||\sqrt{\chi_\lambda}\Delta (v-\Id)||_{L^2(S^2)} \sqrt{\varepsilon}\left( ||\sqrt{\chi_\lambda}\Delta v||_{L^2(S^2)} + ||\sqrt{\chi_\lambda} \Delta \Id||_{L^2(S^2)}\right)
		\\
		&\leqslant c\varepsilon^\frac{3}{2} (\delta +\varepsilon)^\frac{1}{2}(\lambda^2-1).
	\end{align*}
	Choosing $\varepsilon>0$ small enough yields $\lambda=1$. By (\ref{psi w22 est}) $\psi$ must vanish and therefore $v= \Id$ and the M\"obius transformation $M$ must be a rotation. Hence $u$ is a rotation. 
	
\end{proof}

Now we prove the Main Theorem \ref{thm deg one rot}	

\begin{proof}[Proof of Theorem \ref{thm deg one rot}]
	If the statement is not true, there exist sequences $\varepsilon_k\searrow 0$ and critical points $u_{\varepsilon_k}$ with $E_{\varepsilon_k}(u_{\varepsilon_k})\leqslant 12\pi-\mu$ and $\operatorname{deg}(u_{\varepsilon_k})=1$ but $u_{\varepsilon_k}$ is not of the form $u(x)=Rx,~R\in  SO(3)$. With Theorem 1.1 in \cite{lamm06} and Theorem 2 in \cite{duzkuw} we get (up to the choice of a subsequence)
	\begin{align*}
		12\pi > \lim_{k\rightarrow \infty} E_{\varepsilon_k}(u_{\varepsilon_k})= \sum_{k=1}^{m}E(\omega^i)= 4\pi \sum_{i=1}^{m}|\deg (\omega^i)| \qquad\text{and}\qquad \deg(u_{\varepsilon_k})= \sum_{i=1}^{m}\deg(\omega^i) =1,
	\end{align*}
	with $\omega^i$ non-trivial harmonic maps. Therefore  $m=1$ and $\deg (\omega^1)=1$. Since $\omega^1$ is harmonic, $E(\omega^1)=4\pi$. Thus for every $\delta>0$ there exists a $k$ large enough so that 
	\begin{align*}
		E_{\varepsilon_k}(u_{\varepsilon_k})\leqslant 4\pi +\delta \leqslant 4\pi(1+2\varepsilon)+\delta
	\end{align*} 
	and Theorem \ref{main thm 1} implies that $u_{\varepsilon_k}$ is a rotation which, is a contradiction to our assumption.
	The proof follows analogously for maps of degree $-1$.
	
\end{proof}


\section{Gap Theorem for $\varepsilon$-harmonic maps of degree zero}

Now we turn our attention to $\varepsilon$-harmonic maps of degree zero. Theorem \ref{thm deg zero constant} follows analogously to  \cite{lammmalmic19}.  Before we get to the proof, we need a $\varepsilon$-version of the $\alpha$-harmonic  gap theorem of Sacks and Uhlenbeck (\cite{SacksUhlenbeck} Theorem 3.3).  

\begin{lemma}\label{Uhlenbeck gap lemma}
	There exists $\delta,\varepsilon>0$ such that if $u_\varepsilon\in W^{2,2}(S^2,S^2)$ is a critical point of $E_\varepsilon$  and $E_\varepsilon(u_\varepsilon)<\delta $, then $E_\varepsilon(u_\varepsilon)=0$ and $u_\varepsilon$ is a constant map.
\end{lemma}

\begin{proof}
	Let $u_\varepsilon$ be a critical point of $E_\varepsilon$. We multiply the Euler Lagrange equation with $\Delta u_\varepsilon$ and integrate by parts. 
	\begin{align}\nonumber\label{gap uhlenbeck eqn 1}
		\int_{S^2} |\Delta u_\varepsilon|^2 &+\varepsilon |\nabla \Delta u_\varepsilon|^2  dA_{S^2}= \int_{S^2} \langle \Delta u_\varepsilon- \varepsilon \Delta^2 u_\varepsilon, \Delta u_\varepsilon\rangle dA_{S^2}
		\\\nonumber
		&= \int_{S^2} \left\langle -u_\varepsilon|\nabla u_\varepsilon|^2+ \varepsilon u_\varepsilon\bigg( 
		\Delta|\nabla u_\varepsilon|^2 + \Div \langle \Delta u_\varepsilon, \nabla u_\varepsilon\rangle + \langle \nabla \Delta u_\varepsilon, \nabla u_\varepsilon\rangle,\Delta u_\varepsilon\right\rangle dA_{S^2}
		\\\nonumber
		&\leqslant \eta \int_{S^2} |\Delta u_\varepsilon|^2 dA_{S^2} + \varepsilon \eta \int_{S^2} |\nabla \Delta u_\varepsilon|^2dA_{S^2} +c_\eta \int_{S^2} |\nabla u_\varepsilon|^4dA_{S^2} 
		\\
		&\quad + c_\eta \varepsilon^2 \int_{S^2} |\nabla^2 u_\varepsilon|^4 dA_{S^2}.
	\end{align}
	Now we integrate the  Bochner formula (see e.g. Struwe \cite{StruweParkCity})
\begin{align}\nonumber\label{Bochner fomula}
	\Delta_{S^2} \left( \frac{1}{2} |\nabla_{S^2} u|^2\right)&= \p_\mu ( \Delta_{S^2} u^i)\p_\mu u^i + \p^2_{\alpha\mu} u^i \p^2_{\alpha \mu} u^i
	+R_{\alpha\beta} \p_\alpha u^i \p_\beta u^i 
	\\
	&\quad + \frac{1}{2} \p^2_{kl}h_{ij}(u) \p_\mu u^i \p_\mu u^j \p_\alpha u^k \p_{\alpha} u^l.
	\end{align}
	Here we work in local coordinates and $R_{\alpha\beta }$ denotes the Ricci tensor on $S^2$.

 The left-hand side then vanishes, the term involving the Ricci curvature is positive on the sphere and the second derivative of the metric is bounded.  After integrating the first term on the right-hand side by parts we get
	\begin{align*}
		\int_{S^2} |\nabla^2 u_\varepsilon|^2 + |\nabla u_\varepsilon|^2 dA_{S^2} &\leqslant c\int_{S^2} |\Delta u_\varepsilon|^2 + |\nabla u_\varepsilon|^4 dA_{S^2}.
	\end{align*}
	By the Sobolev embedding $W^{1,1}\hookrightarrow L^2(S^2) $ and $E_\varepsilon (u_\varepsilon)\leqslant \delta $ we have 
	\begin{align*}
		\int_{S^2} |\nabla u_\varepsilon|^4 dA_{S^2}&\leqslant c\left(\int_{S^2} |\nabla u_\varepsilon|^2 dA_{S^2}\right)\left(\int_{S^2}|\nabla^2 u_\varepsilon|^2 dA_{S^2}\right) 
		+ c\left(\int_{S^2} |\nabla u_\varepsilon|^2  dA_{S^2}\right)^2
		\\
		&\leqslant c\delta \int_{S^2}|\nabla^2 u_\varepsilon|^2 dA_{S^2} + c\delta \int_{S^2} |\nabla u_\varepsilon|^2  dA_{S^2}
	\end{align*}
	and with $\delta>0 $ small enough
	\begin{align*}
		\int_{S^2} |\nabla^2 u_\varepsilon|^2 + |\nabla u_\varepsilon|^2 dA_{S^2} &\leqslant c\int_{S^2} |\Delta u_\varepsilon|^2dA_{S^2} . 
	\end{align*}	
	Applying all of this to (\ref{gap uhlenbeck eqn 1}) and choosing $\eta$ small enough yields
	\begin{align*}
		\int_{S^2}\varepsilon |\nabla \Delta u_\varepsilon|^2+ |\nabla^2 u_\varepsilon|^2 + |\nabla u_\varepsilon|^2dA_{S^2}&\leqslant c\varepsilon^2\int_{S^2}  |\nabla^2 u_\varepsilon |^4 dA_{S^2}.
	\end{align*}
	Arguing as in the proof of Lemma \ref{chiDeltaPsi small in L2 lemma} we obtain for $\delta, \varepsilon,\eta>0$ small enough 
	\begin{align*}
		\int_{S^2}\left(|\nabla u_\varepsilon|^2 + |\nabla^2 u_\varepsilon|^2 + \varepsilon |\nabla^3 u_\varepsilon|^2 \right)dA_{S^2} &\leqslant c\varepsilon^2\int_{S^2}  |\nabla^2 u_\varepsilon |^4 dA_{S^2}
		\\
		&\leqslant \delta \varepsilon \int_{S^2} \left(|\nabla^3 u_\varepsilon|^2 +|\nabla^2 u_\varepsilon|^2\right)dA_{S^2}.
	\end{align*}
	Hence $E_\varepsilon(u_\varepsilon)=0$ and $u_\varepsilon$ is constant.
	
\end{proof}

\begin{proof}[Proof of Theorem \ref{thm deg zero constant}]
	We assume there exists a sequence $(u_{\varepsilon_k})_{k\in\N}$ of non-constant critical points of $E_{\varepsilon_k}$ with $E_{\varepsilon_k}(u_{\varepsilon_k})\leqslant 8\pi -\delta$. The energy identity for $\epsilon$-harmonic maps  yields
	\begin{align*}
		8\pi> \lim_{k\rightarrow \infty}E_{\varepsilon_k}(u_{\varepsilon_k})= \sum_{i=1}^{N}E(u^i)= \sum_{i=1}^{N} 4\pi |\deg(u^i)|,
	\end{align*}
	where $u^i:S^2\rightarrow S^2,~i=1,...,N$ are  harmonic maps, which are non-trivial for $i\geqslant 2$. With the results of Duzaar and Kuwert \cite{duzkuw} we have
	\begin{align*}
		0=\deg(u_{\varepsilon_k})= \sum_{i=1}^{N}\deg(u^i).
	\end{align*}
	Thus $N=1$ and $u^1$ is a constant harmonic map. Then $\lim_{k\rightarrow \infty}E_{\varepsilon_k}(u_{\varepsilon_k})=0$ and with Lemma \ref{Uhlenbeck gap lemma} it follows that $u_{\varepsilon_k}=const$.
	

\end{proof}
\bigskip
\bigskip
Next we construct explicit examples of $\varepsilon$-harmonic maps of degree zero with $E_\varepsilon(u_\varepsilon)\geqslant 8\pi$ which are not constant.  This shows that the bound in Theorem \ref{thm deg zero constant} is optimal. 
We follow \cite{lammmalmic19} and start by defining a class of rotationally symmetric maps. Let $n\in\N$ and 
\begin{align*}
	[n\pi, (n+1)\pi]\times [0,2\pi] &\rightarrow \R^3\\
	(r,\theta) \qquad&\mapsto (\sin r \, \cos \theta, \ \sin r \, \sin \theta, \ \cos r)
\end{align*}
be a parametrization of $S^2$. For $n$ even, this parametrization is orientation preserving and for $n$ odd orientation reversing. 
Further let $f\in C([0,\pi],\R)$ with 
\begin{align*}
	f(0)=0\qquad \text{and}\qquad f(\pi)=n\pi.
\end{align*}
Then we define $u_f\colon S^2\rightarrow S^2$ by 
\begin{align*}
	u_f\colon [0,\pi]\times [0,2\pi]&\rightarrow \R^3
	\\
	(r,\theta)\qquad&\mapsto \left( \sin(f(r)) \cos(\theta), \sin(f(r)) \sin(\theta) , \cos(f(r)) \right).
\end{align*}	
$u_f$ is rotationally symmetric and wraps $n$ times around $S^2$, reversing orientation after each round. Hence $u_f$ has degree zero if $n$ is even and degree one if $n$ is odd.
\\
\\
Let  $n=2$,
\begin{align*}
	X=\{f\colon [0,\pi] \rightarrow \R : u_f\in W^{2,2}(S^2,\R^3),~ f(0)=0, ~f(\pi)=2\pi \}
\end{align*}
and $M^*=\inf_{f\in X} I(f)$, where 
\begin{align*}
	I(f):= E_\varepsilon(u_f).
\end{align*}
$E_\varepsilon(u_f)$  is invariant under rotations about the $z$-axis and reflections in planes containing the line $(0,0,z)$.  Thus, by the principle of symmetric criticality of Palais (see \cite{Palais79} or Remark 11.4(a) in \cite{amma}),  $f$ is a critical point of $I$ if and only if $u_f$ is a critical point of $E_\varepsilon$. 
We show that there exists $f^*\in X$ with $I(f^*)=M^*$. Let $(f_j)$ be a sequence in $X$ with corresponding sequence $(u_{f_j})\in W^{2,2}(S^2, \R^3)$ and $I(f_j)\searrow M^*$. $(u_{f_j})$ is bounded in $W^{2,2}(S^2,\R^3)$ and thus contains a subsequence (again denoted $u_{f_j}$) with $u_{f_j}\rightharpoonup u_{f^*}$ weakly in $W^{2,2}(S^2,\R^3)$ and uniformly in $C^0(S^2,\R^3)$, with $f^*\in X$. By the lower semi-continuity of $E_\varepsilon$ with respect to weak convergence in $W^{2,2}(S^2,\R^3)$ we have $I(f^*)=E_\varepsilon(u_{f^*})=M^*$.

Now we want to express $E_\varepsilon(u_f)$ in terms of $f$ and compute
\begin{align*}
	\frac{\p  u_f}{\p r}&= f'(r)\bigg( \cos(f(r)) \cos(\theta), \cos (f(r))\sin(\theta), -\sin (f(r)) \bigg)
	\\
	\frac{\p u_f}{\p \theta}&=\bigg(-\sin(f(r)) \sin(\theta), \sin(f(r)) \cos(\theta), 0\bigg).
\end{align*}
Thus we have 
\[ 
\frac{1}{2} |\nabla u_f|^2 = \frac12 \left((f')^2 + \frac{(\sin (f(r)))^2}{(\sin (r))^2}\right). 
\] 
Noting again that 
\[
|\Delta u_f|^2 =|(\Delta u_f)^T|^2+|\nabla u_f|^4 \geqslant |\nabla u_f|^4,
\]
we get a lower bound on $E_\varepsilon(u_{f^*})$
\begin{align*} 
	E_\varepsilon(u_{f^*}) &\geqslant  \pi \int_0^{\pi} \left((f^* \mbox{}')^2 + \frac{(\sin f^*)^2}{(\sin r)^2}\right)\sin r \, dr + \frac{\varepsilon }{2}\int_{S^2} |\nabla u_{f^*}|^4dA_{S^2}
	\\
	&\geqslant 2\pi \int_0^\pi |f^*\mbox{} ' (\sin f^*)| dr + \frac{\varepsilon}{2}\left( 2\pi \int_0^\pi |f^*\mbox{} ' (\sin f^*)| dr \right)^2.
\end{align*} 
There exist $r_1 \in (0, \pi)$ such that $f^*(r_1) = \pi$ and
\begin{align*} 
	\int_0^{\pi} |f^*\mbox{}'(\sin f^*)| \, dr &\geqslant \int_0^{r_1} f^*\mbox{}'(\sin f^*) \, dr - 
	\int_{r_1}^{\pi} f^*\mbox{}' (\sin f^*) \, dr  
	\\ 
	&= \left.-\cos f^*(r)\right|_0^{r_1} + \left.\cos f^*(r)\right|_{r_1}^{\pi} 
	\\ 
	&= 4. 
\end{align*} 
Hence
\begin{align*}
	E_\varepsilon(u_{f^*})\geqslant 8\pi+ 32\varepsilon \pi^2
\end{align*}
and $u_{f^*}$ is a non-constant $\varepsilon$-harmonic map of degree zero.
To complete the proof of Theorem \ref{thm deg zero} we show the following 

\begin{prop}\label{prop deg zero}
	There exists a universal constant $c>0$, such that for any $0<\varepsilon<\frac{1}{4}$
	\begin{align}\label{eqn prop deg zero}
		E_\varepsilon(u_\varepsilon)< 8\pi + c\varepsilon^\frac{1}{2},
	\end{align}
	where  $u_\varepsilon$ is the minimizer of $E_\varepsilon$ among all $u_f$.
\end{prop}

\begin{proof}
	Let $\Lambda>1$ and 
	\begin{align*}
		f(r)=\begin{cases}
			2\arctan(\Lambda\tan (r)),\qquad &0\leqslant r\leqslant \frac{\pi}{2},
			\\
			2\arctan(\Lambda\tan (r))+2\pi,\qquad &\frac{\pi}{2}<r\leqslant \pi.
		\end{cases}
	\end{align*}
	We consider the corresponding map $u_f\in X$. As $r$ increases from $0$ to $\frac{\pi}{2}$, $f(r)$ increases from $0$ to $\pi$, which means that $u_f$ maps the upper hemisphere to the full sphere with the equator being mapped to the south pole $(0,0,-1)$.  As $r$ increases from $\frac{\pi}{2}$ to $\pi$,  $f(r)$ increases from $\pi$ to $2\pi$, which means that $u_f$ maps the lower hemisphere to the full sphere but with opposite orientation. Thus $u_f$ has degree zero.  We want to estimate $E_\varepsilon(u_f)$. An explicit calculation (which can be found in section $4.5$ of \cite{Jasi}) gives
\begin{align*}
		\frac{1}{2}\int_{S^2} |\Delta u_f|^2 dA_{S^2}
		&=  32\pi\int_0^1 \Lambda^{-6} \frac{t^2(1-t^2)}{(1-a^2t^2)^4}dt +  32\pi \int_0^1\Lambda^{-4} \frac{(1+t^2)^2}{(1-a^2t^2)^4}dt,
	\end{align*}
	where $a^2=1-\Lambda^{-2}$.
	We now estimate each term separately. First we note that $t\in[0,1]$ and $0\leqslant a\leqslant 1$
	\begin{align*}
		\frac{t^2(1-t^2)}{(1-a^2t^2)^4}\leqslant \frac{(a^2-a^2t^2)}{a^2 (1-a^2t^2)^4}\leqslant \frac{1}{ a^2(1-a^2t^2)^3}=\frac{1}{a^2(1+at)^3(1-at)^3}\leqslant \frac{1}{a^2(1-at)^3}.
	\end{align*}
	Then
	\begin{align*}
		32\pi\int_0^1 \Lambda^{-6} \frac{t^2(1-t^2)}{(1-a^2t^2)^4}dt &\leqslant 32\pi \Lambda^{-6} \int_0^1 \frac{1}{a^2(1-at)^3}dt 
		= 32\pi \Lambda^{-6} \left( \frac{1}{2a^3(1-a)^2} -\frac{1}{2a^3}\right)
		\\
		&\leqslant 32\pi \Lambda^{-6} \frac{(1+a)^2}{a^2(1-a^2)^2}
		\leqslant 128\pi \Lambda^{-2}\frac{1}{a^2}= 128\pi \frac{\Lambda^{-2}}{1-\Lambda^{-2}}= 128\pi \frac{1}{\Lambda^2-1},
	\end{align*}
	where we used that $\Lambda^{-2}=1-a^2$. Analogously we get for the second  term 
	\begin{align*}
		\frac{(1+t^2)^2}{(1-a^2t^2)^4}\leqslant \frac{(1+t)^4}{(1+at)^4(1-at)^4}= \frac{(a+at)^4}{a^4(1+at)^4(1-at)^4}
		\leqslant \frac{1}{a^4(1-at)^4}
	\end{align*}
	and therefore 
	\begin{align*}
		32\pi \int_0^1\Lambda^{-4} \frac{(1+t^2)^2}{(1-a^2t^2)^4}dt&\leqslant 32\pi \Lambda^{-4}\int_0^1 \frac{1}{a^4(1-at)^4}dt =32\pi \Lambda^{-4}\left( \frac{1}{3a^5(1-a)^3}- \frac{1}{3a^5}\right)
		\\
		&\leqslant 32\pi \Lambda^{-4} \frac{(1+a)^3}{a^4(1-a^2)^3}\leqslant 256 \pi \Lambda^2\frac{1}{a^4}= 256\pi \frac{\Lambda^2}{(1-\Lambda^{-2})^2}
		\\
		&= 256\pi \frac{\Lambda^{6}}{(\Lambda^2-1)^2}. 
	\end{align*}
	All in all we get 
	\begin{align*}
		\frac{1}{2}\int_{S^2} |\Delta u_f|^2 dA_{S^2}&\leqslant  128\pi \frac{1}{\Lambda^2-1}+ 256\pi \frac{\Lambda^{6}}{(\Lambda^2-1)^2} . 
	\end{align*}
	Analogously we estimate the first part of $E_\varepsilon(u_f)$ (see \cite{lammmalmic19} Proposition 3.2) 
	\begin{align}\nonumber
		\frac{1}{2}\int_{S^2} |\nabla u_f|^2 dA_{S^2}&= 8\pi  \int_0^\frac{\pi}{2}\Lambda^2 \frac{1+\cos^2 r}{(\cos^2(r) +\Lambda^2 \sin^2(r))^2}\sin rdr
		\\\nonumber
		&= 8\pi \int_0^1 \Lambda^{-2} \frac{(1+t^2)}{(1-a^2t^2)^2}dt
		\\\nonumber
		&\leqslant 8\pi \int_0^1 \frac{\Lambda^{-2}}{2a^2}\left( \frac{1}{(1-at)^2}+\frac{1}{(1+at)^2}\right)dt 
		\\\nonumber
		&= 8\pi \Lambda^{-2} \frac{1}{a^2(1-a^2)} = 8\pi \frac{1}{a^2}= 8\pi \frac{\Lambda^2}{\Lambda^2-1}.
	\end{align}		
	Together with the above we get 
	\begin{align*}
		\frac{1}{2}\int_{S^2} \bigg( |\nabla u_f|^2 +\varepsilon |\Delta u_f|^2 \bigg) dA_{S^2}&\leqslant 8\pi \frac{\Lambda^2}{\Lambda^2-1} + 128\pi\varepsilon \frac{1}{\Lambda^2-1}+ 256\pi\varepsilon \frac{\Lambda^{6}}{(\Lambda^2-1)^2}  .
	\end{align*}
	We choose $\Lambda^2>2$, then  $ \frac{\Lambda^2}{\Lambda^2-1}<1+2\Lambda^{-2} < 2$ and thus 
	\begin{align*}
		E_\varepsilon(u_f)< 8\pi \left(1+\frac{2}{\Lambda^2}\right) + 128\pi\varepsilon+ 1024 \pi \varepsilon \Lambda^2 .
	\end{align*}
	We set $\Lambda:=\varepsilon^{-\frac{1}{4}}$. Note that for $\varepsilon\in (0, \frac{1}{4})$, $\Lambda^2>2$ still holds. Then we get 
	\begin{align*}
		E_\varepsilon(u_f)&< 8\pi \left( 1+ 2\varepsilon^\frac{1}{2}\right)+ 1152 \pi \varepsilon^\frac{1}{2}
		= 8\pi +1168\pi \varepsilon^\frac{1}{2}.
	\end{align*}
	Since $u_\varepsilon$ minimizes $E_\varepsilon$  among all maps in $X$, we have 
	\begin{align*}
		E_\varepsilon(u_\varepsilon)\leqslant E_\varepsilon(u_f)<8\pi +c\varepsilon^\frac{1}{2},
	\end{align*}
	with $c=1168\pi$.
	
\end{proof}

\begin{proof}[Proof of Theorem \ref{thm deg zero}]
	Given $\delta>0$ we choose $\varepsilon\in (0,\frac{1}{4})$ such that $\varepsilon<(\frac{\delta}{c})^2$, where $c$ is the constant in (\ref{eqn prop deg zero}). With Proposition \ref{prop deg zero} and Theorem \ref{thm deg zero constant} we get 
	\begin{align*}
		8\pi+32\varepsilon\pi^2\leqslant  E_\varepsilon (u_\varepsilon) < 8\pi + c \varepsilon^\frac{1}{2} <8\pi +\delta.
	\end{align*}
\end{proof}

\end{document}